\documentclass[english]{article}
\usepackage{float}
\usepackage{algorithm,algorithmic}
\usepackage{amsmath}
\usepackage{amssymb}
\usepackage{graphicx}
\usepackage{bm}
\usepackage{amsthm}
\usepackage{subcaption}
\usepackage{enumitem}
\usepackage[title]{appendix}

\usepackage{geometry}
\newgeometry{
	textheight=9in,
	textwidth=5.5in,
	top=1in,
	bottom = 1.1in,
	headheight=12pt,
	headsep=25pt,
	footskip=30pt,
	lmargin=.8in,rmargin=.8in
}

\usepackage[utf8]{inputenc} 
\usepackage[T1]{fontenc}    
\usepackage{url}            
\usepackage{booktabs}       
\usepackage{amsfonts}       
\usepackage{nicefrac}       
\usepackage{microtype}  
\usepackage{mathrsfs}

\numberwithin{equation}{section}
\usepackage{hyperref}
\hypersetup{
	colorlinks=true,
	linkcolor=red,
	filecolor=magenta,      
	citecolor=blue,
}
\newtheorem{lem}{Lemma}
\newtheorem{prop}{Proposition}

\newtheorem{coro}{Corollary}

\numberwithin{theo}{section}
\numberwithin{lem}{section}
\numberwithin{prop}{section}
\numberwithin{alg}{section}
\numberwithin{assum}{section}
\numberwithin{defi}{section}
\numberwithin{coro}{section}

\theoremstyle{remark}
\newtheorem{remark}{Remarks}
\numberwithin{remark}{section}

\usepackage{stfloats}
\usepackage{wrapfig}

\usepackage{babel}

\usepackage{mathtools}
\def\Tiny{\fontsize{4pt}{4pt}\selectfont}
\newcommand*{\eqdef}{\ensuremath{\overset{\mathclap{\text{\Tiny def}}}{=}}}

\def\Tiny{\fontsize{4pt}{4pt}\selectfont}

\title{Towards Scalable Semidefinite Programming: Optimal Metric ADMM with A Worst-case Performance Guarantee
} 

\author{ 	
	Yifan Ran,  Stefan  Vlaski, and Wei Dai
}
\date{}
\begin{document}
	\maketitle
	\makeatletter{\renewcommand*{\@makefnmark}{}
		\footnotetext{
			The	authors are with the department of Electrical and Electronics Engineering, Imperial College London, UK (e-mail: \{y.ran18,  s.vlaski, wei.dai1\}@imperial.ac.uk).  \makeatother}}

	\begin{abstract}
		Despite  the numerous uses of semidefinite programming (SDP) and its universal solvability  via interior point methods  (IPMs),
		it is rarely applied to practical large-scale problems. This mainly owes to the  computational cost of IPMs  that increases  in a bad exponential  way with the   data size.  
		While first-order algorithms such as ADMM can alleviate this issue, but the scalability improvement appears far not enough. 
		In this work,  we aim to achieve extra acceleration for ADMM by appealing to a  non-Euclidean metric space, while maintaining   everything in closed-form expressions.
		The efficiency gain comes from  the  extra degrees of freedom of a  variable metric compared to a scalar step-size, which allows us  to capture some additional ill-conditioning structures.

		
		On the  application side, we consider the  quadratically constrained quadratic program (QCQP), which  naturally appears in  an  SDP form after  a  dualization procedure. 
		This technique, known as  semidefinite relaxation,   has important  uses  across different fields,  particularly  in wireless communications. 
		Numerically, we observe that  the scalability property is  significantly  improved. 
		Depending  on the data generation process, the  extra  acceleration can easily surpass the  scalar-parameter   efficiency  limit, and  the advantage is rapidly increasing  as  the data  conditioning becomes worse.  
		
	
	\end{abstract}
	
	\vspace{5pt}
	\noindent 
	$ \textbf{Keywords:} $\, Semidefinite programming,  Variable metric methods, 
	Preconditioning,  Parameter selection, Duality,
	  Quadratically constrained quadratic program,  Alternating  Direction Method of Multipliers (ADMM)

	\section{Introduction}
	Semidefinite programming (SDP)  \cite{vandenberghe1996semidefinite,boyd1994linear,vandenberghe1995primal,kamath1992continuous,kamath1993nl} is widely recognized as one  of the most  important breakthroughs in the last century. 
	It is developed as a generalization to linear programming (LP), and   is not much harder to solve  \cite{vandenberghe1996semidefinite}.
	The solvability issue is perfectly addressed  by the  interior-point methods  (IPMs), which were first introduced  in 1984 by Karmarkar \cite{karmarkar1984new} for LP,
	and later  extended to all   convex  programs by  Nesterov and Nemirovsky in 1988 \cite{nesterov1988general,nesterov1990optimization,nesterov1992conic,nesterov1994interior}.
	 Many well-known convex solvers such as  CVX \cite{grant11}, and  MOSEK \cite{aps2019mosek} are built on the  IPMs. 
	Despite the great successes, the IPMs  are not well-suited to a newly arisen  challenge, `Big Data'. 
	This is due to its   second-order nature that involves Hessian information. The Hessian is in general a dense matrix  even when the data is highly structured. The computational  cost  is therefore  in  general  prohibitively  expensive for  large-scale data.
	In the literature,   first-order algorithms (FOAs) that only use  gradient (or subgradient) information are considered one of the  most promising tools  for large-scale problems, see a comprehensive survey paper by  Amir Beck \cite{beck2017first}, and a simplified  view  by Marc Teboulle \cite{teboulle2018simplified}.


	Among the  large  class of FOAs,  ADMM has  received an  increasing  amount of  attention. This partly owes  to its  outstanding practical performance and  also    elegant theory, see a survey paper by   Stephen Boyd et. al. \cite{boyd12}.
	The ADMM  algorithm  itself is not  new,  originally proposed by Glowinski and Marrocco \cite{glowinski1975approximation} in 1975, and Gabay and Mercier \cite{gabay1976dual} in 1976. 
	It is known to be equivalent to  some  popular algorithms from other fields. 
	 The perhaps most important one is the Douglas-Rachford Splitting (DRS) \cite{lions1979splitting, douglas1956numerical, peaceman1955numerical,  eckstein1992douglas}  in numerical  analysis.  Due to their  equivalence, the ADMM is typically analysed by converting to the DRS, and invoking the fixed-point theory   \cite{bauschke2017convex, ryu2022large}.
	Most recently, Yifan  Ran established a direct fixed-point analysis for ADMM (that does  not  need DRS), see \cite[sec. 6]{ran2023equilibrate}.
	Additionally, the Primal-Dual Hybrid Gradient (PDHG) method \cite{chambolle2011first, esser2010general, pock2009algorithm}  is also shown to be equivalent to ADMM, see  O'Connor and  Vandenberghe \cite{o2020equivalence}.
	Apart from being equivalent, some strong connections are known  between ADMM and Spingarn’s method of partial inverses, Dykstra’s alternating projections method, Bregman iterative algorithms for $ l_1 $ problems in signal	processing,  see more details from  \cite{boyd12}.

%

	In the literature, different FOAs for solving  SDP  mainly
	 differ  in  how  the positive semidefinite (PSD)  constraint is handled. 
	One typical approach is by noting that any PSD matrix  can be characterized by $ \bm{R}^T\bm{R} $. 
	Hence, the PSD  requirement is removed if the variable is substituted to  $ \bm{R} $, see related  works in  \cite{burer2003nonlinear, burer2005local,wang2023decomposition}. 
	It is  worth  noticing  that the efficiency  of this approach highly depends  on  the  dimension  of  $ \bm{R} $.  Indeed,  if the underlying solution is low-rank,  $ \bm{R} $ admits very few columns and  the problem could  be solved efficiently.
	Another way is to enforce  the PSD requirement by  a new   projected variable,  i.e.,  $ \Pi_{\mathbb{S}^{N}_+}  (\bm{X})$, where $ \Pi_{\mathbb{S}^{N}_+} $  is a projector onto the PSD cone, equivalent to setting all negative eigenvalues to zeros.
	 Owing  to such a projector being  strongly semi-smooth \cite{sun2002semismooth}, an inexact semi-smooth Newton-CG method is designed, see \cite{zhao2010newton}.
	The last approach is  via ADMM, which handles the PSD  requirement  by introducing an  auxiliary variable and performs a projection $ \Pi_{\mathbb{S}^{N}_+}$ separately, see a typical success  in \cite{wen2010alternating}.

	Despite the  successes of FOAs, they  share a common weakness  --- their performances highly depend on the conditioning of the data. 
	The  technique  for  addressing this issue is  known  as  \textit{preconditioning}.  A typical heuristic to select a preconditioner is by appealing to the condition number, where, given a linear system $ \bm{Ax}  = \bm{b} $, one rewrites it into $ \bm{EAx}  = \bm{Eb}  $, and  selects   the \textit{preconditioner} $ \bm{E} $  such that the  new  matrix  $ \bm{EA} $   has  a  smaller condition number  (the largest  eigenvalue divide  the smallest eigenvalue).
	This approach is  well-known and perhaps the most popular one. 
	However, we emphasize that it is in nature a heuristic. 
	That said,  even when the condition number is minimized, there is no guarantee that the algorithm convergence rate will  be improved. 
	Indeed,  how to select the preconditioner to optimize the convergence  rate   is open, see e.g.  \cite[sec. 1.2, First-order methods]{stellato2020osqp},   \cite[sec.8, Parameter selection]{ryu2022large}.
	Due to the lack of optimal choice, the  preconditioner is typically manually tuned in practice. 
	However, manual tuning is problematic. First, for a  matrix preconditioner, there exist multiple parameters, depending  on the degrees of freedom.  How to simultaneously tune them seems to be a huge obstacle. Even when tuning a scalar, the fine-tuned choice is likely not transferable to other types of data, or simply different data sizes.
	Indeed,  we  note that   in \cite{zhao2010newton}, the authors fix  the step-size parameter  to either $ 1 $ or $ 1.6 $, without adapting it to different input data.
	
	Most recently, such an optimal parameter issue  has been addressed by Yifan Ran, where the author optimizes a   general worst-case convergence  rate without  strong assumptions, 
	see the scalar case in  \cite{ran2023general} and the metric space case in  \cite{ran2023equilibrate}.  
	In either  case, the optimal parameter choice  	can be determined by solving a degree-4 polynomial. Except,  
	  in the scalar setting, a closed-form optimal solution always exists  regardless of the initialization choice; 
	in a  non-Euclidean metric space, despite  the  potential  extra acceleration,   a  closed-form expression for both  the metric choice and the algorithm iterates are not guaranteed. 
	
	In this work, we consider the metric space case  in order to gain  extra acceleration. We will address the closed-form issue. To our best knowledge, there is no metric-type ADMM  applied  to SDP in the literature. This is likely due to the aforementioned projection step $ \Pi_{\mathbb{S}^{N}_+}$ no longer works in a non-Euclidean space, i.e., it is not a feasible ADMM iterate.
	 Indeed, we will  show that, given a special class of (decomposed) metrics $  \mathcal{S} $,  the  feasible projection step would be
	$ \mathcal{S}^{-1}\Pi_{\mathbb{S}^{N}_+}\mathcal{S} $.

		On the  application side, we consider the  quadratically constrained quadratic program (QCQP). It is NP-hard and non-convex in nature.
		A  typical  approach, known as \text{semidefinite relaxation}, can reformulate it into an SDP. This technique plays a fundamental role in wireless communications, see a survey paper \cite{luo2010semidefinite}  and references therein. 
		We will present  the first metric-type ADMM solver for the convexified QCQP and numerically evaluate its performance.

	\subsection{Notations}
	The Euclidean space  is denoted by $  \mathscr H $   with inner product  $ \langle \cdot, \, \cdot \rangle   $ equipped, and its induced norm $ \Vert \cdot \Vert $. We denote  its  extension to a metric space as $  \mathscr {H}_{\mathcal M} $,
	with norm $\Vert\bm{v} \Vert_{\mathcal{M}} = \sqrt{\langle \bm{v},  \mathcal{M}\bm{v}\rangle} $.
	We denote by $  \Gamma_0 (\mathscr H) $  the  space  of  convex, closed  and proper  (CCP) functions from $ \mathscr H $  to the extended real line  $ ]-\infty, +\infty]$,
	by $ \mathbb{S}^N  $, the  space of $ N $-dimensional symmetric matrices,
	by $ \mathbb{S}^N_+  $, the  space of $ N $-dimensional positive semidefinite matrices.
	At last, the uppercase bold, lowercase bold, and not bold letters are used for matrices, vectors, and scalars, respectively. The uppercase calligraphic letters, such  as $ \mathcal{A} $ are used to denote  operators.

	\subsection{ADMM algorithm}
	The ADMM algorithm solves the following general convex problem:
	\begin{align}\label{ADMMpro}
	&\,\underset{\bm{x},\bm{z}}{\text{minimize}}\quad   f(\bm{x}) + g(\bm{z})  \nonumber\\
	&\text{subject\,to}\quad  \mathcal{A}\bm{x} - \mathcal{B}\bm{z} = \bm{c} ,
	\end{align}
with functions $ f, g \in \Gamma_0  (\mathscr H)$ and operators $ \mathcal{A},\, \mathcal{B} $ being injective. Moreover,  we assume a solution    exists.

	\subsubsection{Basic  iterates}
	The  basic ADMM iterates are based on the following augmented Lagrangian: 
	\begin{equation}\label{L_gamma}
	\mathcal{L}_\gamma(\bm{x},\bm{z},\bm{\lambda}) 
	\,=\,  f(\bm{x}) + g (\bm{z})  + \frac{\gamma}{2}\Vert\mathcal{A}\bm{x} - \mathcal{B}\bm{z} - \bm{c} + \bm{\lambda}/\gamma\Vert^2,  
	\end{equation}	
	where $\gamma > 0$ denotes a positive step-size.	 The ADMM iterates are 
	\begin{align}\label{admm_sca}
	\bm{x}^{k+1} =\,\,& \underset{\bm{x}}{\text{argmin}} \,\, 	\mathcal{L}_\gamma(\bm{x},\bm{z}^k,\bm{\lambda}^k)  \nonumber\\				
	\bm{z}^{k+1} =\,\, & \underset{\bm{z}}{\text{argmin}} \,\, 	\mathcal{L}_\gamma(\bm{x}^{k+1},\bm{z},\bm{\lambda}^k) \nonumber\\
	\bm{\lambda}^{k+1} =\,\,  &\bm{\lambda}^{k} + \gamma( \mathcal{A}\bm{x}^{k+1} - \mathcal{B}\bm{z}^{k+1} - \bm{c}),	  \tag{scalar ADMM}
	\end{align}	
	which are guaranteed to converge (if minimizers  exist)  with any  positive step-size $\gamma > 0$.
	
	\subsubsection{Metric space  iterates}
	The above $\gamma$-parametrized iterates can be extended to a  metric space environment $\mathscr H_\mathcal{M}$,   with the following augmented Lagrangian: 
\begin{equation}
\mathcal{L}_\mathcal{M}(\bm{x},\bm{z},\bm{\lambda})
\,=\,   f(\bm{x}) + g (\bm{z})  + \frac{1}{2}\Vert \mathcal{A}\bm{x} - \mathcal{B}\bm{z}  + \mathcal{M}^{-1}\bm{\lambda}\Vert^2_{\mathcal M},  
\end{equation}
with  norm $\Vert\bm{v} \Vert_{\mathcal{M}} = \sqrt{\langle \bm{v},  \mathcal{M}\bm{v}\rangle} $.
 The ADMM iterates are 
	\begin{align}\label{admm_met}
	\bm{x}^{k+1} =\,\,& \underset{\bm{x}}{\text{argmin}} \,\, 	\mathcal{L}_\mathcal{M}(\bm{x},\bm{z}^k,\bm{\lambda}^k)  \nonumber\\				
	\bm{z}^{k+1} =\,\, & \underset{\bm{z}}{\text{argmin}} \,\, 	\mathcal{L}_\mathcal{M}(\bm{x}^{k+1},\bm{z},\bm{\lambda}^k) \nonumber\\
	\bm{\lambda}^{k+1} =\,\,  &\bm{\lambda}^{k} + \mathcal{M}( \mathcal{A}\bm{x}^{k+1} - \mathcal{B}\bm{z}^{k+1} - \bm{c}), 	  \tag{metric ADMM}
	\end{align}	
	which are guaranteed to converge (if minimizers  exist)  with any  positive definite metric  $\mathcal{M}  \succ 0$.


\subsection{Semidefinite programming}

Semidefinite programming (SDP) considers minimizing a linear function of variable $ \bm{x} \in\mathbb{R}^m $ subject to a Linear Matrix  Inequality (LMI). 
The standard form is  given below, see  e.g.   \cite{vandenberghe1996semidefinite}:
\begin{align}\label{pri_0}
\underset{\bm{x}}{\text{minimize}}\quad  &    \langle\bm{c}, \bm{x}\rangle\nonumber\\
\text{subject\,to}\quad &   \bm{A}_0 + \sum_{i=1}^m  x_i\bm{A}_i\succeq 0. \tag{Primal}
\end{align}
with  $ \bm{A}_i \in \mathbb{S}^{N}, \,\,  i= 0, 1, \dots, m  $  being symmetric matrices.
Its Fenchel dual problem  is 
\begin{align}\label{dual_0}
\qquad\qquad\,\,  \underset{\bm{X}}{\text{minimize}}\quad\,\,\,  &   \langle\bm{A}_0, \bm{X}\rangle  \nonumber\\
\qquad\qquad\,\,	\text{subject\,to}\quad\,\,\,				   &  \langle\bm{A}_i, \bm{X}\rangle  = c_i, \quad  i = 1,\dots, m \nonumber\\
&    \,\,   \bm{X} \succeq 0.  \tag{Dual}
\end{align}

\section{Preliminary: a unified treatment}
To  start, we show that \eqref{pri_0}  and \eqref{dual_0} share  the same intrinsic structure, 
and can be  handled in a  unified way, via
\begin{align}\label{sdp_uni}
&\,\underset{\bm{X},\bm{Z}}{\text{minimize}}\quad\,   f\,(\bm{X}) +  \delta_{\mathbb{S}^{N}_+} (\bm{Z})   \qquad \nonumber\\
&\text{subject\,to}\quad  \mathcal{A}\,  \big(\bm{X}\big) = \bm{Z}. 						\qquad\tag{Unified}
\end{align}
For \eqref{pri}, the matrix variable $ \bm{X} $ reduces to a vector $ \bm{x}	 $, and
the following definitions are used: 
\begin{align}\label{1}
\quad(\text{Function}) \qquad\,\,\,	&\,f\,(\bm{x}) \,\eqdef \,\,    \langle\bm{c}, \bm{x}\rangle	, \qquad 	 \text{dom}\, f =  \mathscr H 	\nonumber\\
\,\,(\text{Constraint} )  \qquad 	 & \mathcal{A}\,  \big(\bm{x}\big) \eqdef  \bm{A}_0 + \sum_{i=1}^m  x_i\bm{A}_i.
\end{align}
For \eqref{dual}, the following definitions are used: 
\begin{align}\label{2}
\quad(\text{Function}) \qquad\,\,\,	&f\,(\bm{X}) \,\,\eqdef    \langle\bm{A}_0, \bm{X}\rangle , \quad  \text{dom}\, f =  \{ \bm{X} \in \mathbb{S}^{N} \,\, |\,\, \langle\bm{A}_i, \bm{X}\rangle = c_i,  \,\, \forall i  \} 		\nonumber\\
\,\,(\text{Constraint} ) \qquad 	 &\mathcal{A}\,  \big(\bm{X}\big) \eqdef  \bm{X}.
\end{align}

Recall the abstract ADMM iterates in \eqref{admm_sca}, we can apply it to  the  above SDP   \eqref{sdp_uni} to obtain specific closed-form iterates:
\begin{align}\label{sdp_admm}
\bm{X}^{k+1} =\,\,& \underset{\bm{X}}{\text{argmin}} \,\,\,	f(\bm{X})  + \frac{\gamma}{2}\Vert\mathcal{A}\bm{X} - \bm{Z}^k  + \bm{\Lambda}^k/\gamma\Vert^2  \nonumber\\				
\bm{Z}^{k+1} =\,\, &\,  \Pi_{\mathbb{S}^{N}_+}  \, \big(\mathcal{A}\bm{X}^{k+1} + \bm{\Lambda}^{k}/\gamma  \big) \nonumber\\
\bm{\Lambda}^{k+1} =\,\,  &\,\bm{\Lambda}^{k} + \gamma  \big( \mathcal{A}\bm{X}^{k+1} - \bm{Z}^{k+1} \big),	 \tag{scalar solver}
\end{align}	
where $ \Pi_{\mathbb{S}^{N}_+}  $ is a projector onto the positive-definite cone $ \mathbb{S}^{N}_+ $,  which is  equivalent to setting all  the negative eigenvalues  (of the input) to zeros.
Let us note that $ f $ is a linear function, with either full domain as in  \eqref{1}, or domain being the  boundary of a polyhedron as in \eqref{2}. Both cases, the evaluation admits  a  closed-form expression. 
 Overall, we see that all iterates can be efficiently evaluated in such a scalar setting.

\subsection{Primal:  differentiation \& injectivity}
Here, we discuss some detailed issues for  the above primal setting \eqref{1}.   We note that an additional manipulation is  needed there to obtain the first-order information. 
Specifically, we need to rewrite the LMI into
\begin{align}\label{mani}
\mathcal{A}\,  \big(\bm{x}\big) =  \bm{A}_0 + \sum_{i=1}^m  x_i\bm{A}_i =  \bm{A}_0 +   \text{mat} \bigg( \widetilde{\bm{A}}\bm x \bigg),
\end{align}
with 
\begin{equation}\label{A_til}
\widetilde{\bm{A}} \eqdef [ \text{vec}(\bm{A}_1) , \cdots, \text{vec}(\bm{A}_m)  ]  \, \in \, \mathbb{R}^{N^2 \times m}.
\end{equation}
where  $  \text{mat}(\cdot) $ and $  \text{vec}(\cdot) $ denote  reshaping the input into a matrix, and  a vector, respectively, and  they are  inverse operations.
This enables a manipulation on the Euclidean norm term. That is, given a matrix input $ \bm  V $,  
\begin{equation}
\Vert 	\mathcal{A}\bm{x}	-   \bm  V \Vert^2 
= \Vert  \text{vec} \big( \mathcal{A}\bm{x}	-   \bm  V  \big)	\Vert^2
=  \Vert  \text{vec} \big( \bm{A}_0 -   \bm  V \big)+  \widetilde{\bm{A}}\bm x	\,\Vert^2
\end{equation}
which yields the following first-order  optimality condition:
\begin{align}\label{x_first}
\bm x^\star = - \big( \widetilde{\bm{A}}^T\widetilde{\bm{A}}  \big)^{-1}\widetilde{\bm{A}}^T \text{vec} \big( \bm{A}_0 -   \bm  V \big).
\end{align}

Moreover, in view of \eqref{x_first}, we see a natural requirement that  $  \widetilde{\bm{A}}^T\widetilde{\bm{A}}  $ should be  invertible, i.e.,   $ \widetilde{\bm{A}} $ has full column-rank (injective).  
That is,  we need all columns in  \eqref{A_til} not  linearly  dependent, which  is in fact implicitly  guaranteed.
This  is most clear from the dual setting \eqref{2}. 
Suppose  there exists such a linear dependency. Then,  constraints $  \langle\bm{A}_i, \bm{X}\rangle = c_i,  \,\, \forall i  $ will  contain  either duplicated or contradicted ones (depends on $ c_i $),  implying an ill-posed problem. We  should  omit this  trivial  case.

%

\subsection{Towards extra efficiency }\label{sec_chall}
By now, we have seen  that all  things  are  quite elegant in the scalar-parameter setting.
However, its efficiency appears still  not enough, as we find its iteration number increases quite significantly with the data size, particularly in some ill-conditioning settings. This motivates us to further exploit the  efficiency by appealing to a metric-parameter setting, as in \eqref{admm_met}. Due to the extra degree of freedoms in a variable metric, in principle we should be able to gain extra acceleration.
This is true, except some new challenges are accompanied:\\

$\bullet$\, (i) In a metric environment $\mathscr H_\mathcal{M}$, the  $ \bm{Z} $-update no  longer admits a closed-form expression as  in \eqref{sdp_admm};\\

$\bullet$\, (ii) Due to the extra degree of freedoms, tuning a metric parameter is much harder than  a scalar.\\

For the first obstacle (i), without a closed form, the algorithm efficiency will  dramatically decrease. For (ii), without knowing the optimal choice, one needs to manually  tune multiple  parameters simultaneously.  
We note that a randomly generated metric parameter is typically much worse than a random scalar parameter, implying the tuning difficulty. 
Even  one managed to obtain a fine-tuned choice, it is likely needs retune when the data setting  changes, either different data or  simply different sizes.

Indeed, these are two significant obstacles. 
As a consequence, despite  the promising extra acceleration of the metric-type ADMM, to  our best knowledge, it is not  employed for the SDP in the literature. 
In this work, we will address these two challenges. 
For challenge (i), we prove that if an additional `definiteness invariance' condition holds, then a closed-form expression can be obtained, see Prop. \ref{pro_clo}. 
Moreover, by appealing  to Schur-complement  lemma, we are able to construct  metrics satisfying this special condition.
The second  obstacle (ii) is addressed by fixing the metric to its optimal choice. Quite remarkably, given our constructed metrics, their optimal choices are guaranteed in closed-form expressions (under zero initialization).  
Furthermore, when equipping our metrics, the resulted solver is guaranteed no worse than \eqref{sdp_admm}, in the worst-case convergence rate sense. Additionally, we identify the data structure that would cause the worst situation.

	\section{Closed-form guaranteed   metrics}
	In the previous section, we discussed two obstacles for employing a metric parameter.  Here, we address the first challenge  there, the closed-form iterate  issue.
	Our success relies on employing a special class of metrics, that do not change the definiteness property of a matrix variable.
	
	\subsection{Definiteness invariant condition}
	Here, we  present what-we-call the  definiteness invariant condition. When it holds,  we prove that a closed-form expression can be obtained (in a non-Euclidean metric space).
	\begin{prop}\label{pro_clo}
		Consider a  metric space  environment $\mathscr H_\mathcal{M}$, with $ \langle \cdot, \,  \mathcal{M}\,\cdot \rangle = \langle \mathcal{S}\,\cdot, \,  \mathcal{S}\,\cdot \rangle  $. 	Suppose the decomposed metric $ \mathcal{S} $ admits the following definiteness-invariant characterization:
		\begin{equation}\label{de_inv}
		\mathcal{S}(\bm{Z}) \in \mathbb{S}^{N}_+   \,\, \Longleftrightarrow \,\,      \bm{Z} \in \mathbb{S}^{N}_+,  \tag{invariance}
		\end{equation}
				with $  \bm{Z} \in  \mathbb{S}^{N} $	being an arbitrary symmetric matrix,  		
 where  $ \mathbb{S}^{N}_+  = \{ \bm{X} \in \mathbb{S}^{N} \, |\,   \bm{X} \succeq 0  \} $ denotes the  positive  semidefinite cone.

		Then, 
		\begin{equation}\label{closed}
		\mathcal{S}^{-1}\circ\Pi_{\mathbb{S}^{N}_+ }\circ\mathcal{S} \, (\bm{V})  = \underset{\bm{Z}}{\text{argmin}} \,\, \delta_{\mathbb{S}^{N}_+ } (\bm{Z}) +  \frac{1}{2}\Vert \bm{Z} - \bm{V}\Vert^2_{\mathcal{M}}.
		\end{equation}
	\end{prop}
	\begin{proof}
		We   aim  to prove relation \eqref{closed}. From its right-hand  side,  we  obtain
		\begin{align}\label{to_cont}
		\,\underset{\bm{Z}}{\text{argmin}} \,\, \delta_{\mathbb{S}^{N}_+ } (\bm{Z}) +  \frac{1}{2}\Vert \bm{Z} - \bm{V}\Vert^2_{\mathcal{M}}	
		=& \,\underset{\bm{Z}}{\text{argmin}} \,\, \delta_{\mathbb{S}^{N}_+ } (\bm{Z}) +  \frac{1}{2}\Vert \mathcal{S}\bm{Z} - \mathcal{S}\bm{V}\Vert^2,	\nonumber\\			
		=& \,\mathcal{S}^{-1}\, \underset{\widetilde{\bm{Z}}}{\text{argmin}} \,\, \delta_{\mathbb{S}^{N}_+ } (\mathcal{S}^{-1}\widetilde{\bm{Z}}) +  \frac{1}{2}\Vert \widetilde{\bm{Z}} - \mathcal{S}\bm{V}\Vert^2,
		\end{align}
		with $  \widetilde{\bm{Z}}= \mathcal{S}\bm{Z} $. By \eqref{de_inv}, we have
		\begin{equation}
		\widetilde{\bm{Z}}\in \mathbb{S}^{N}_+  \,\,\Longleftrightarrow\,\,  \mathcal{S}^{-1}(\widetilde{\bm{Z}}) \in \mathbb{S}^{N}_+. 
		\end{equation}
		It follows that
		\begin{equation}
		\delta_{\mathbb{S}^{N}_+ } (\mathcal{S}^{-1}\widetilde{\bm{Z}}) = \begin{cases}
		0  &  \widetilde{\bm{Z}} \in \mathbb{S}^{N}_+, \\
		+\infty & \text{otherwise},
		\end{cases} 
		\end{equation}
		which reduces to $ \delta_{\mathbb{S}^{N}_+ } (\widetilde{\bm{Z}}) $. Equation \eqref{to_cont} can therefore be written as
		\begin{equation}
		\mathcal{S}^{-1}\underset{\widetilde{\bm{Z}}}{\text{argmin}} \,\, \delta_{\mathbb{S}^{N}_+ } (\widetilde{\bm{Z}}) +  \frac{1}{2}\Vert \widetilde{\bm{Z}} - \mathcal{S}\bm{V}\Vert^2\nonumber\\
		= \,\mathcal{S}^{-1}\circ\Pi_{\mathbb{S}^{N}_+ }\circ\mathcal{S} \,(\bm{V}).
		\end{equation}	
		The proof is now concluded.
	\end{proof}


	\subsection{Construction: feasible choices}
	Here, we aim  to construct metrics such  that \eqref{de_inv} holds. 	
	Consider the following construction strategy:	
	\begin{prop}\label{S_def_compact}
		Given $ \gamma_1,\gamma_2>0 $ and integer $ K $ from $ \{ 1, 2,  \dots, N-1 \} $. Let the variable  metric $  \mathcal{M} $ be defined as
		\begin{equation}\label{stan_a}
		\mathcal{M} \,(\bm{V}) \eqdef  
		\left[\begin{array}{cc} 
		\frac{\gamma_1}{\gamma_2}\,\bm{1}_1 &   \qquad\gamma_1\,\, \bm{1}_0   
		\\ \,\gamma_1\, \bm{1}_0 &  \,\quad\gamma_1\gamma_2 \,\, \bm{1}_2  \end{array}\right] \, \odot\, \bm{V}, \tag{construction}
		\end{equation}
		with $ \bm{1}_1 \in  \mathbb{S}^K $, $ \bm{1}_0 \in  \mathbb{R}^{K\times  (N-K)} $,  $ \bm{1}_2 \in  \mathbb{S}^{N-K }$,
		where by $ \bm{1} $ we denote the ones matrix (i.e., all  entries being $ 1 $), by  $ \odot $  the  element-wise multiplication.
	\end{prop}	
For the above metrics, we can prove that condition \eqref{de_inv} always holds, and consequently a closed-form evaluation result  obtained. 
	\begin{prop}
		Let metric $  \mathcal{M} $ be defined as in \eqref{stan_a}.
		Then, condition \eqref{de_inv} holds.
	\end{prop}
	\begin{proof}
	Given an arbitrary symmetric matrix $ \bm{X} \in \mathbb{S}^N $, it can be partitioned into
		\begin{equation}
		\bm{X} = \left[\begin{array}{cc}\bm{X}_1 &   \bm{X}_0  \\ \bm{X}_0^T  &  \bm{X}_2 \end{array}\right],
		\end{equation}
		with $ \bm{X}_1 \in  \mathbb{S}^K $, $ \bm{X}_0 \in  \mathbb{R}^{K\times  (N-K)} $,  $ \bm{X}_2 \in  \mathbb{S}^{N-K }$, where  $ K $ is an arbitrary integer from $ \{ 1, 2,  \dots, N-1 \} $.
		By  the generalized  \textit{Schur complement} argument, see \cite[A.5.5]{boyd2004convex}, we arrive at
		\begin{equation}\label{Schur_complement} 
		\left[\begin{array}{cc} \sqrt{\frac{\gamma_1}{\gamma_2}}\, \bm{X}_1 &  \quad \sqrt{\gamma_1}\,\bm{X}_0  \\ \:\sqrt{\gamma_1}\,\bm{X}_0^T  &  \,\sqrt{\gamma_1\gamma_2}\,\, \bm{X}_2 \end{array}\right]
		\succeq 0 \Longleftrightarrow 
		\sqrt{\frac{\gamma_1}{\gamma_2}}\bm{X}_{1}  \succeq 0,\:\,   \sqrt{\gamma_1\gamma_2}(\bm{X}_{2}-\bm{X}_{0}^\text{T}\bm{X}_{1}^\dagger \bm{X}_{0}) \succeq 0,\:\,  \sqrt{\gamma_1}(\bm{I}-\bm{X}_{1}\bm{X}_{1}^\dagger)\bm{X}_{0}=0,
		\end{equation}
		where $\cdot^\dagger $ denotes the pseudo-inverse. 	
		In  view  of the  right-hand side above,   we note that \\
		
		$\bullet$\,(i) $ \sqrt{\frac{\gamma_1}{\gamma_2}}\,\,\bm{X}_{1} \succeq 0     \,\iff\,   \bm{X}_{1} \succeq 0$; \\
		
		$\bullet$\,(ii) $ \sqrt{\gamma_1\gamma_2} \,\big(\bm{X}_{2}-\bm{X}_{0}^\text{T}\bm{X}_{1}^\dagger \bm{X}_{0}\big) \succeq 0 \,\iff\,   \bm{X}_{2}-\bm{X}_{0}^\text{T}\bm{X}_{1}^\dagger \bm{X}_{0} \succeq 0$;\\
		
		$\bullet$\,(iii) $ \sqrt{\gamma_1}\,\big(\bm{I}-\bm{X}_{1}\bm{X}_{1}^\dagger\big)\bm{X}_{0}=0 \,\iff\,    \big(\bm{I}-\bm{X}_{1}\bm{X}_{1}^\dagger\big)\bm{X}_{0}=0$. \\
		
		Combing  the  above 3  relations, we obtain
		 \begin{equation}\label{cond2}
		 \mathcal{S}(\bm{X}) \succeq 0 \,\iff\, \bm{X} \succeq 0,
		 \end{equation}
		which coincides with condition  \eqref{de_inv}. The proof is  therefore concluded.
	\end{proof}

	\begin{remark}
		A feasible metric or preconditioner that satisfies \eqref{de_inv} is  first  found  by Yifan  Ran et.  al.  in \cite[sec.  3.1]{yifan_icassp}. However, due to the lack of  an optimal choice, it did not draw interest at the time.
	\end{remark}

	\subsection{Closed-form  metric solver}
	Given any variable metric  $ \mathcal{M}   =  \mathcal{S}^*\mathcal{S} $ defined via  \eqref{stan_a}. The  SDP  \eqref{sdp_uni} can be solved by 
	 the following  closed-from iterates:
	\begin{align}
	\bm{X}^{k+1} =\,\,& \underset{\bm{X}}{\text{argmin}} \,\,\,	f(\bm{X})  + \frac{1}{2}\Vert\mathcal{A}\bm{X} - \bm{Z}^k  + \mathcal{M}^{-1}\bm{\Lambda}^k\Vert^2_\mathcal{M} , \nonumber\\				
	\bm{Z}^{k+1} =\,\, &   	\mathcal{S}^{-1}\Pi_{\mathbb{S}^{N}_+}\mathcal{S}   \bigg(\mathcal{A}\bm{X}^{k+1} + \mathcal{M}^{-1}\bm{\Lambda}^{k}  \bigg)    ,\nonumber\\
	\bm{\Lambda}^{k+1} =\,\,  &\bm{\Lambda}^{k} + \mathcal{M}  \bigg( \mathcal{A}\bm{X}^{k+1} - \bm{Z}^{k+1} \bigg).	 \tag{metric solver}
	\end{align}	

Moreover, we may simplify the above iterates by employing scaled variables.
	\begin{align}\label{sdp_scaled}
\bm{X}^{k+1} =\,\,& \underset{\bm{X}}{\text{argmin}} \,\,\,	f(\bm{X})  + \frac{1}{2}\Vert\mathcal{S}\mathcal{A}\bm{X} - \widetilde{\bm Z}^k  + \widetilde{\bm \Lambda}^k\Vert^2 , \nonumber\\				
\widetilde{\bm Z}^{k+1} =\,\, &   	\Pi_{\mathbb{S}^{N}_+}\, \bigg( \widetilde{\bm X}^{k+1} + \widetilde{\bm{\Lambda}}^{k}  \bigg) ,\nonumber\\
\widetilde{\bm\Lambda}^{k+1} =\,\,  & \widetilde{\bm\Lambda}^{k} +  \widetilde{\bm X}^{k+1} - \widetilde{\bm  Z}^{k+1} ,	   \tag{scaled form}
\end{align}	
with the following variable  substitutions:
\begin{equation}
\widetilde{\bm X}^{k} = \mathcal{S}\mathcal{A}\bm{X}^{k}, \quad
 \widetilde{\bm Z}^{k} =  \mathcal{S}\bm{Z}^{k}, \quad \widetilde{\bm{\Lambda}}^{k} = (\mathcal{S}^*)^{-1}\bm{\Lambda}^{k}.
\end{equation}
It may worth noticing a nice property of \eqref{sdp_scaled},
that   $ \bm{X} $-update is not scaled. That said, the final solution $ \bm{X}^\star $ is exactly the problem solution, and there is no need to `scale back'.


	\section{Optimal metric selection }
Here,  we  determine the optimal choice of the variable metric, corresponding to challenge (ii) in Sec. \ref{sec_chall}.  The general optimal  metric selection rule  is most recently established by   Yifan Ran \cite{ran2023equilibrate},  based on optimizing the following convergence rate bound:	
\begin{lem}\cite[Theorem 1]{ryu2022large} \label{lem_rate}
	Assume $\mathcal{F}: \mathbb{R}^n \rightarrow \mathbb{R}^n $ is $\theta$-averaged with $ \theta \in\, ]0,1[ $ and  $\text{Fix}\,\, \mathcal{F} \neq \emptyset $. Then, $ \bm{\zeta}^{k+1} = \mathcal{F}{\bm{\zeta}^{k}} $ with any starting point $ \bm{\zeta}^0 \in \mathbb{R}^n $ converges to one fixed-point, i.e., 
	\begin{equation}
	\bm{\zeta}^k \rightarrow \bm{\zeta}^\star
	\end{equation}
	for some $ \bm{\zeta}^\star \in \text{Fix}\,\, \mathcal{F}$. The quantities $ \text{dist}\,\, (\bm{\zeta}^k , \text{Fix}\,\, \mathcal{F}) $, $ \Vert \bm{\zeta}^{k+1} - \bm{\zeta}^k  \Vert $, and $ \Vert \bm{\zeta}^k  - \bm{\zeta}^\star \Vert $ for any $ \bm{\zeta}^\star \in \text{Fix}\,\, \mathcal{F}$ are monotonically non-increasing with $ k $. Finally, we have
	\begin{equation}
	\text{dist}\,\, (\bm{\zeta}^k , \text{Fix}\,\, \mathcal{F}) \rightarrow 0,
	\end{equation}
	and
	\begin{equation}\label{rate}
	\Vert \bm{\zeta}^{k+1} - \bm{\zeta}^k  \Vert^2 \leq \frac{\theta}{(k+1)(1-\theta)}	\text{dist}^2\,\, (\bm{\zeta}^0 , \text{Fix}\,\, \mathcal{F}).
	\end{equation}
\end{lem}
Intuitively,  one can directly  optimize the above upper bound  by simply substituting the ADMM fixed-point.
One expression (a dual view) is at least trace back to 1989 by Eckstein \cite{eckstein1989splitting}.
 However, not too  surprisingly,  a direct substitution will not work (otherwise not  a long-standing open issue). 
As pointed out in  \cite{ran2023equilibrate}, there exists an implicit, parameter-related scaling in the upper bound, introduced by the classical way of parametrization.
It is only optimizable after removing the extra scaling. This leads  to the following result: 
	\begin{coro}[SDP choice]
	For the ADMM algorithm solving \eqref{sdp_uni}, the  optimal choice of  (positive definite)  metric $ \mathcal M   =  \mathcal S^* \mathcal S  $ can  be found via
	\begin{align}\label{gen}
	\underset{\mathcal S  }{\text{minimize}}\,\,  	 \Vert \mathcal{S}\mathcal{A}\bm{X}^\star\Vert^2  +  \Vert (\mathcal{S}^*)^{-1} \bm{\Lambda}^\star    \Vert^2 - 2 \langle\mathcal{S}\mathcal{A}\bm{X}^\star, \bm{\zeta}^0 \rangle -    2 \langle(\mathcal{S}^*)^{-1} \bm{\Lambda}^\star,  \bm{\zeta}^0 \rangle,
	\end{align}
 where $  \bm{\zeta}^0 = \mathcal{A}\bm{X}^0  + \bm{\Lambda}^0 $   is an arbitrary  initialization.	
	\end{coro}
Throughout the rest of the paper,  we will limit our discussion to zero initialization $  \bm{\zeta}^0   =  0 $, since otherwise we cannot guarantee a closed-form optimal choice.

	\subsection{Structure I: optimal parameters}
	Given a specific  metric  definition, we can substitute it to  \eqref{gen} to find the optimal choice. Recall our definition in \eqref{stan_a}, we obtain the following result:
	\begin{prop}
		Let  metric $ \mathcal{M} $ be defined as in \eqref{stan_a}. For ADMM solving  SDP \eqref{sdp_uni}, the optimal choice of $ \mathcal{M} $ under zero initialization can be found by
		\begin{equation}\label{opt_pro_specific}
		\underset{ \gamma_1,\gamma_2>0 }{\text{minimize}}\quad
		\frac{\gamma_1}{\gamma_2}\Vert \bm{X}^\star_1 \Vert^2 +  \frac{\gamma_2}{\gamma_1}\Vert \bm{\Lambda}^\star_1 \Vert^2
		+ \gamma_1\gamma_2\Vert \bm{X}^\star_2 \Vert^2 + 
		\frac{1}{\gamma_1\gamma_2}\Vert \bm{\Lambda}^\star_2 \Vert^2 + 2\gamma_1\Vert \bm{X}^\star_0 \Vert^2 + \frac{2}{\gamma_1}\Vert \bm{\Lambda}^\star_0 \Vert^2, \tag{$ \mathcal{P}_1 $}
		\end{equation}
		where the following partition definitions are used: $ \bm{X} \eqdef \left[\begin{array}{cc} \bm{X}_1 & \bm{X}_0\\ \bm{X}_0^\text{T} & \bm{X}_2 \end{array}\right],  \,\, \bm{\Lambda} \eqdef \left[\begin{array}{cc} \bm{\Lambda}_1 & \bm{\Lambda}_0\\ \bm{\Lambda}_0^\text{T} & \bm{\Lambda}_2 \end{array}\right]$.	
	\end{prop}
		\begin{proof}
			For $  \bm{\zeta}^0 =  0 $,  \eqref{gen} reduces to
			\begin{equation}
				\underset{\mathcal S  }{\text{minimize}}\,\,  	 \Vert \mathcal S\bm{X}^\star\Vert^2  +  \Vert (\mathcal{S}^*)^{-1} \bm{\Lambda}^\star    \Vert^2
			\end{equation} 
			Then, substituting the definition in \eqref{stan_a} concludes the  proof. 
		\end{proof}

	Next, we show that the optimal choice is unique.
	\begin{prop}\label{prop_blk}
		Problem \ref{opt_pro_specific} admits a unique solution pair $ (\gamma_1^\star,\gamma_2^\star) $, with  $ \gamma_2^\star $ being the unique positive real root of the following degree-4 polynomial:
		\begin{equation} \label{quartic}
		\gamma_2^4 \Vert \bm{X}^\star_2\Vert^2\Vert \bm{\Lambda}^\star_1 \Vert^2 + \gamma_2^3 (\Vert \bm{X}^\star_2\Vert^2\Vert \bm{\Lambda}^\star_0 \Vert^2 + \Vert \bm{X}^\star_0 \Vert^2\Vert \bm{\Lambda}^\star_1 \Vert^2) 
		-  \gamma_2 (\Vert \bm{\Lambda}^\star_2 \Vert^2\Vert \bm{X}^\star_0 \Vert^2 + \Vert \bm{\Lambda}^\star_0 \Vert^2 \Vert \bm{X}^\star_1 \Vert^2) - \Vert \bm{\Lambda}^\star_2 \Vert^2 \Vert \bm{X}^\star_1 \Vert^2 = 0, 
		\end{equation}
		and $ \gamma_1^\star $ being:
		\begin{equation}\label{rho1_closed}
		\gamma_1^\star  = \sqrt{\frac{ {\gamma_2^\star}\Vert\bm{\Lambda}^\star_1 \Vert^2 + \frac{1}{\gamma_2^\star}\Vert\bm{\Lambda}^\star_2 \Vert^2 + 2\Vert \bm{\Lambda}^\star_0 \Vert^2}{\frac{1}{\gamma_2^\star}\Vert\bm{X}^\star_1 \Vert^2 + \gamma_2^\star\Vert \bm{X}^\star_2\Vert^2 + 2\Vert \bm{X}^\star_0 \Vert^2}}.
		\end{equation}
	\end{prop}
	\begin{proof}
	First,  we show the solution uniqueness. Let us  note that the objective function is  a  Euclidean norm  square  function, and the domain is the positive orthant. That said, we are minimizing   a strictly convex function, its minimizer  is therefore unique, owing to \cite{boyd2004convex}. 
	
	Now, we establish the closed-form expression for the solution pair  $ (\gamma_1^\star,\gamma_2^\star) $.
The minimizer is obtained when the two gradients w.r.t. $\gamma_1$ and $\gamma_2$ vanish, i.e., 
		\begin{align}
		\frac{1}{{\gamma_2^\star}}		& \Vert \bm{X}^\star_1 \Vert^2   -  \frac{{\gamma_2^\star}}{{\gamma_1^\star}^2}\Vert \bm{\Lambda}^\star_1 \Vert^2
		+ {\gamma_2^\star}\Vert \bm{X}^\star_2 \Vert^2 - \frac{1}{{\gamma_1^\star}^2{\gamma_2^\star}}\Vert \bm{\Lambda}^\star_2 \Vert^2 + 2\Vert \bm{X}^\star_0 \Vert^2 - \frac{2}{{\gamma_1^\star}^2}\Vert \bm{\Lambda}^\star_0 \Vert^2  =   0 , \label{eq1}\\	 
		-\frac{{\gamma_1^\star}}{{\gamma_2^\star}^2} &\Vert \bm{X}^\star_1 \Vert^2 +  \frac{1}{{\gamma_1^\star}}\Vert \bm{\Lambda}^\star_1 \Vert^2
		+ {\gamma_1^\star}\Vert \bm{X}^\star_2 \Vert^2 - 
		\frac{1}{{\gamma_1^\star}{\gamma_2^\star}^2}\Vert \bm{\Lambda}^\star_2 \Vert^2 = 0. 
		\end{align}
		By \eqref{eq1}, we instantly obtain the $\gamma_1^\star$ expression in  our proposition.
		
		All what left  is the  $\gamma_2^\star$ expression. To show  it, since  $\gamma_1, \gamma_2 > 0$, we can rewrite the above two relations into
		\begin{align}
		{\gamma_1^\star}^2 & \Vert \bm{X}^\star_1 \Vert^2   -  {{\gamma_2^\star}^2}\Vert \bm{\Lambda}^\star_1 \Vert^2
		+ {\gamma_1^\star}^2{\gamma_2^\star}^2\Vert \bm{X}^\star_2 \Vert^2 - \Vert \bm{\Lambda}^\star_2 \Vert^2 + 2{\gamma_1^\star}^2{\gamma_2^\star}\Vert \bm{X}^\star_0 \Vert^2 - {2}{\gamma_2^\star}\Vert \bm{\Lambda}^\star_0 \Vert^2  =   0,  \\	 
		-{\gamma_1^\star}^2 &\Vert \bm{X}^\star_1 \Vert^2 +  {\gamma_2^\star}^2\Vert \bm{\Lambda}^\star_1 \Vert^2
		+ {\gamma_1^\star}^2 {\gamma_2^\star}^2\Vert \bm{X}^\star_2 \Vert^2 - 
		\Vert \bm{\Lambda}^\star_2 \Vert^2 = 0, 
		\end{align}
		which, by addition and subtraction,  can be simplified to 
		\begin{align}
		{\gamma_1^\star}^2&{\gamma_2^\star}^2\Vert \bm{X}^\star_2 \Vert^2 - \Vert \bm{\Lambda}^\star_2 \Vert^2 + {\gamma_1^\star}^2{\gamma_2^\star}\Vert \bm{X}^\star_0 \Vert^2 - {\gamma_2^\star}\Vert \bm{\Lambda}^\star_0 \Vert^2  =   0 , \\	 
		{\gamma_1^\star}^2 & \Vert \bm{X}^\star_1 \Vert^2   -  {{\gamma_2^\star}^2}\Vert \bm{\Lambda}^\star_1 \Vert^2
		+ {\gamma_1^\star}^2{\gamma_2^\star}\Vert \bm{X}^\star_0 \Vert^2 - {\gamma_2^\star}\Vert \bm{\Lambda}^\star_0 \Vert^2  =   0.
		\end{align}
		Separating $ {\gamma_1^\star} $ gives 
		\begin{align}\label{rho1_forms}
		{\gamma_1^\star}^2  = \frac{\Vert \bm{\Lambda}^\star_2 \Vert^2 + {\gamma_2^\star}\Vert \bm{\Lambda}^\star_0 \Vert^2}{{\gamma_2^\star}^2\Vert \bm{X}^\star_2\Vert^2 + {\gamma_2^\star}\Vert \bm{X}^\star_0 \Vert^2},  \quad  	 
		{\gamma_1^\star}^2  =  \frac{{\gamma_2^\star}^2\Vert \bm{\Lambda}^\star_1 \Vert^2 + {\gamma_2^\star}\Vert \bm{\Lambda}^\star_0 \Vert^2}{\Vert \bm{X}^\star_1 \Vert^2 + {\gamma_2^\star}\Vert \bm{X}^\star_0 \Vert^2} .  \tag{$ \gamma_1^\star $}
		\end{align}
		The above two relations must hold simultaneously, yielding the  following degree-4 polynomial:
		\begin{equation} \label{poly}
		{\gamma_2^\star}^4 \Vert \bm{X}^\star_2\Vert^2\Vert \bm{\Lambda}^\star_1 \Vert^2 + {\gamma_2^\star}^3 (\Vert \bm{X}^\star_2\Vert^2\Vert \bm{\Lambda}^\star_0 \Vert^2 + \Vert \bm{X}^\star_0 \Vert^2\Vert \bm{\Lambda}^\star_1 \Vert^2) 
		-  {\gamma_2^\star} (\Vert \bm{\Lambda}^\star_2 \Vert^2\Vert \bm{X}^\star_0 \Vert^2 + \Vert \bm{\Lambda}^\star_0 \Vert^2 \Vert \bm{X}^\star_1 \Vert^2) - \Vert \bm{\Lambda}^\star_2 \Vert^2 \Vert \bm{X}^\star_1 \Vert^2 = 0.
		\end{equation}
		The proof is now concluded.
	\end{proof}

	 Luckily, degree-4 is the highest order  polynomial that admits a closed-form solution, first proved in the Abel–Ruffini theorem in 1824. 
	 Lodovico Ferrari is considered the one to first find the solution in 1540. 
	This is an ancient topic that is widely known, except that  the closed-form formula often appears messy in the literature. The simplest and correct one we found is in  \cite{quartic}.  Particularly, due to the quadratic coefficient is zero in our case, we made some simplifications.

	For light of notations,  the quartic polynomial in \eqref{quartic} can be abstractly cast as
	\begin{equation} 
	a \gamma_2^4 + b\gamma_2^3 +  d\gamma_2 + e = 0.
	\end{equation}
	For  real coefficients and $ b, d $ not simultaneously equal 0, the above polynomial admits the following four roots:
	\begin{equation}\label{closed_rho2}
	\gamma_2 = \begin{cases}
	\frac{1}{2} ( -\frac{b}{2a} - {u_4} - \sqrt{u_5 - u_6})  , 		\\
	\frac{1}{2} ( -\frac{b}{2a} - {u_4} + \sqrt{u_5 - u_6})  ,		\\
	\frac{1}{2} ( -\frac{b}{2a} + {u_4} - \sqrt{u_5 + u_6}) , 	\\
	\frac{1}{2} ( -\frac{b}{2a} + {u_4} + \sqrt{u_5 + u_6})  ,
	\end{cases} 
	\end{equation}
	where
	\begin{equation}
	u_4 = \sqrt{\frac{b^2}{4a^2}+u_3},\quad
	u_5 = \frac{b^2}{2a^2} - u_3,\quad
	u_6 = -\frac{\frac{b^3}{a^3} + \frac{8d}{a} }{4u_4},
	\end{equation}
	and where
	\begin{equation}
	u_1 = \frac{\sqrt{27}}{2}(ad^2 + b^2  e),\quad
	u_2 = u_1+ \sqrt{(bd - 4ae)^3 + u_1^2},\quad
	u_3 = \frac{1}{\sqrt{3}a} (\sqrt[3]{u_2} - \frac{bd - 4ae}{ \sqrt[3]{u_2}} ).	
	\end{equation}

	\subsection{Structure II: partition selection}
		In the previous section, we determined the optimal  choice of any feasible metric. 
		We emphasize that there are multiple feasible definitions, depending on the partition integer $ K $, recall \eqref{S_def_compact}. 
		Then, if the partitioning changes,  the coefficients of the degree-4 polynomial will also change. 
	Since  $ K $ can be any integer from  $ \{ 1, 2,  \dots, N-1 \} $, there exist $ (N-1) $ different partition options and consequently
		 $ (N-1) $ different polynomials.	
		This implies another type of degree-of-freedom  that can be exploited. 
		We will characterize it here.

%

	For notation convenience, denote the  collection of all associated partitioned matrices by
	\begin{equation}
	\mathbb P^N = \bigg\{ \left[\begin{array}{cc} 
	\bm{V}_1\,\, \in  \mathbb{S}^K 		\qquad\quad\,\, 									&   \qquad\bm{V}_0   \in  \mathbb{R}^{K\times  (N-K)}\\ 
	 \bm{V}_0^\text{T}  \in  \mathbb{R}^{ (N-K)\times K}	& \bm{V}_2 \in \mathbb{S}^{N-K } \end{array}\right]  \,\, \bigg\vert \,\, \forall\,\bm{V} \in \mathbb{S}^N , \,  K = 1, \dots, N-1  \bigg\}.    \tag{partition}
	\end{equation}
	Then, following the basic polynomial in \eqref{opt_pro_specific}, we may additional consider the partitioning, which yields
	\begin{equation}\label{pro0}
	\underset{  \bm{X}^\star, \bm\Lambda^\star  \in  \mathbb P^N }{\text{minimize}}\quad
	\underset{ \gamma_1,\gamma_2>0 }{\text{minimize}}\quad
	\frac{\gamma_1}{\gamma_2}\Vert \bm{X}^\star_1 \Vert^2 +  \frac{\gamma_2}{\gamma_1}\Vert \bm{\Lambda}^\star_1 \Vert^2
	+ \gamma_1\gamma_2\Vert \bm{X}^\star_2 \Vert^2 + 
	\frac{1}{\gamma_1\gamma_2}\Vert \bm{\Lambda}^\star_2 \Vert^2 + 2\gamma_1\Vert \bm{X}^\star_0 \Vert^2 + \frac{2}{\gamma_1}\Vert \bm{\Lambda}^\star_0 \Vert^2.
	\end{equation}
	The inner problem is solvable, enabling a simplification.
	\begin{lem}
	Problem	\eqref{pro0} can be simplified  into
		\begin{equation}\label{pro1}
	\underset{  \bm{X}^\star, \bm\Lambda^\star  \in  \mathbb P^N }{\text{minimize}}\quad
	 2\sqrt{ \bigg(	\frac{1}{\gamma_2^\star}\Vert \bm{X}^\star_1 \Vert^2  + \gamma_2^\star\Vert \bm{X}^\star_2 \Vert^2 + 2\Vert \bm{X}^\star_0 \Vert^2 \bigg)
	 	\bigg( \frac{1}{\gamma_2^\star}\Vert \bm{\Lambda}^\star_2 \Vert^2 +  {\gamma_2^\star}\Vert \bm{\Lambda}^\star_1 \Vert^2 + 2\Vert \bm{\Lambda}^\star_0 \Vert^2\bigg)},
		\end{equation}
		with $ \gamma_2^\star $  given by the unique positive real root of  \eqref{quartic}.
	\end{lem}
	\begin{proof}
		Following Prop. \ref{prop_blk},   $ ({\gamma_1^\star},{\gamma_2^\star}) $ is the unique solution pair of  the inner problem  in \eqref{pro0}. 
		Hence, its minimum can be attained by substituting the optimal solution, yielding
	\begin{align}
f =& \quad \frac{\gamma_1^\star}{\gamma_2^\star}\Vert \bm{X}^\star_1 \Vert^2 +  \frac{\gamma_2^\star}{\gamma_1^\star}\Vert \bm{\Lambda}^\star_1 \Vert^2
+ \gamma_1^\star\gamma_2^\star\Vert \bm{X}^\star_2 \Vert^2 + 
\frac{1}{\gamma_1^\star\gamma_2^\star}\Vert \bm{\Lambda}^\star_2 \Vert^2 + 2\gamma_1^\star\Vert \bm{X}^\star_0 \Vert^2 + \frac{2}{\gamma_1^\star}\Vert \bm{\Lambda}^\star_0 \Vert^2 \nonumber\\
=&\quad {\gamma_1^\star}\bigg(	\frac{1}{\gamma_2^\star}\Vert \bm{X}^\star_1 \Vert^2  + \gamma_2^\star\Vert \bm{X}^\star_2 \Vert^2 + 2\Vert \bm{X}^\star_0 \Vert^2 \bigg) +  \frac{1}{\gamma_1^\star} \bigg( \frac{1}{\gamma_2^\star}\Vert \bm{\Lambda}^\star_2 \Vert^2 +  {\gamma_2^\star}\Vert \bm{\Lambda}^\star_1 \Vert^2 + 2\Vert \bm{\Lambda}^\star_0 \Vert^2\bigg) \nonumber\\
=&\quad   2\sqrt{ \bigg(	\frac{1}{\gamma_2^\star}\Vert \bm{X}^\star_1 \Vert^2  + \gamma_2^\star\Vert \bm{X}^\star_2 \Vert^2 + 2\Vert \bm{X}^\star_0 \Vert^2 \bigg)\bigg( \frac{1}{\gamma_2^\star}\Vert \bm{\Lambda}^\star_2 \Vert^2 +  {\gamma_2^\star}\Vert \bm{\Lambda}^\star_1 \Vert^2 + 2\Vert \bm{\Lambda}^\star_0 \Vert^2\bigg)},
\end{align}	
where the  last line is  by invoking \eqref{rho1_closed}. The proof is now concluded.
	\end{proof}
	
Next, we can substitute the explicit closed-form expression of $ \gamma_2^\star $, recall \eqref{closed_rho2}. However, this  yields a complicated problem, and appears no closed-form solution available. 
This motivates us to consider an alternative. 
We note that the partition options are finite with  $ (N-1) $ choices. Recall that given one choice, there is a corresponding polynomial. Owing to the polynomial admitting a unique closed-form solution, an exhaustive search over the $ (N-1) $ candidates are computationally feasible. 
This leads to a search strategy  detailed  in  the  next section.

		\subsection{Full-structure optimal choice}
	Following the previous section, we do a finite $ (N-1) $-times search to  exploit the partition structure. 
	The result algorithm yields  a solution to the joint optimization problem \eqref{pro0}, or equivalently \eqref{pro1}.
	\begin{algorithm}[H]
		\caption{A finite search }
		\label{Algo_1}
		\begin{algorithmic}[1]
			\REQUIRE Optimal primal and dual solutions $ \bm{X}^\star , \bm{\Lambda}^\star $.
			\WHILE{ $ K = 1, \cdots, N-1 $ } 
			\STATE Denote
			\begin{equation}
			\bm{X} = \left[\begin{array}{cc} \bm{X}_1 & \bm{X}_0\\ \bm{X}_0^\text{T} & \bm{X}_2 \end{array}\right],  \quad
			\bm{\Lambda} = \left[\begin{array}{cc} \bm{\Lambda}_1 & \bm{\Lambda}_0\\ \bm{\Lambda}_0^\text{T} & \bm{\Lambda}_2 \end{array}\right],
			\end{equation}	
			with  $ \bm{X}_1,\bm{\Lambda}_1 \in \mathbb{S}^{K} $,
			\, $ \bm{X}_0,\bm{\Lambda}_0 \in \mathbb{R}^{K \times (N-K)} $, 
			\,\,$ \bm{X}_2,\bm{\Lambda}_2 \in \mathbb{S}^{N-K} $.	
			\STATE Compute  $ \gamma_2^\star $, as the  unique positive  real  root  of 
			\begin{equation*} 
			{\gamma_2^\star}^4 \Vert \bm{X}^\star_2\Vert^2\Vert \bm{\Lambda}^\star_1 \Vert^2 + {\gamma_2^\star}^3 (\Vert \bm{X}^\star_2\Vert^2\Vert \bm{\Lambda}^\star_0 \Vert^2 + \Vert \bm{X}^\star_0 \Vert^2\Vert \bm{\Lambda}^\star_1 \Vert^2) 
			-  {\gamma_2^\star} (\Vert \bm{\Lambda}^\star_2 \Vert^2\Vert \bm{X}^\star_0 \Vert^2 + \Vert \bm{\Lambda}^\star_0 \Vert^2 \Vert \bm{X}^\star_1 \Vert^2) - \Vert \bm{\Lambda}^\star_2 \Vert^2 \Vert \bm{X}^\star_1 \Vert^2 = 0,
			\end{equation*}	
			with a   closed-form expression in \eqref{closed_rho2}.
			\STATE Compute
			\begin{equation}
			\gamma_1^\star = \frac{\alpha}{\beta}, \quad\,\, f_K = 2{\alpha}{\beta},
			\end{equation}		
			with $ \alpha = \sqrt{\frac{1}{\gamma_2^\star}\Vert \bm{X}^\star_1 \Vert^2  + \gamma_2^\star\Vert \bm{X}^\star_2 \Vert^2 + 2\Vert \bm{X}^\star_0 \Vert^2} $, \,\,\, 
			$ \beta = \sqrt{\frac{1}{\gamma_2^\star}\Vert \bm{\Lambda}^\star_2 \Vert^2 +  {\gamma_2^\star}\Vert \bm{\Lambda}^\star_1 \Vert^2 + 2\Vert \bm{\Lambda}^\star_0 \Vert^2} $.
			\ENDWHILE
			\ENSURE 	Choose  $ (\gamma_1^\star,\gamma_2^\star) $ corresponding to  $\underset{K}{\text{argmin}}\, \{f_K\}$.
		\end{algorithmic}
	\end{algorithm}

	\begin{remark}[Optimal  partition guess]
	While theoretically we need to check all $ (N-1) $ partition choices,  in  practice it  is possible to make a guess beforehand. 
	Empirically, we find that the optimal partition corresponds to the most significant ill-conditioning structure within the optimal solutions.	
	For example, in our later  applications,  
	the QCQP  typically admits the optimal partitioning being $ K =  N-1 $, due to the right-bottom block $ \bm{\Lambda}_2 $ always fixed to $ 1 $, causing  a pre-known ill-conditioning structure.  
	\end{remark}

	\section{Performance guarantee}\label{peform}
	In this section, we provide a performance guarantee by investigating the worst-case scenario.
	 Specifically,  we
	(i)  show that our constructed metric  is no worse than a scalar parameter; 
	(ii)  identify  the structure that leads to the  worst  case.

	\subsection{Worst-case  performance}
	We claim that the worst case is when the variable metric reduces to a scalar parameter. 
	To see this, let us note that the variable metric  $ \mathcal{M} $  in \eqref{stan_a} can be  decomposed into
	\begin{equation}\label{step_size2}
	\mathcal{M}
	= \mathcal{M}_2 \circ \mathcal{M}_1,
	\end{equation}	
	with
	\begin{equation}
	\mathcal{M}_1  = \gamma_1 \mathcal{I},
	\quad
	\mathcal{M}_2  = 
	\left[\begin{array}{cc} 
	\frac{1}{\gamma_2}\,\bm{1}_1 &   \qquad\, \bm{1}_0   
	\\ \quad\,\, \bm{1}_0 &  \,\quad\gamma_2 \, \bm{1}_2  \end{array}\right] , 
	\end{equation}
	where $ \mathcal{I} $ denotes the identity operator. 
	From above, we  see that $ \gamma_1 $ plays the exact same role as  a scalar step-size, and  $ \gamma_2 $ arises owing to the   \textit{Schur complement} lemma. It instantly follows  that, \\
	
		$\bullet$\,(i) If  $ \gamma_2 = 1 $, the variable  metric reduces to a scalar parameter;  \\
		
		$\bullet$\,(ii) Regardless of the partitioning, the joint  optimal pair $ (\gamma_1^\star , \gamma_2^\star ) $ cannot be no worse than the partial optimal choice $ (\gamma_1^\star, 1 ) $. \\


	In fact, $ \gamma_2^\star = 1 $ is exactly the worst  case, as illustrated below. 
	\begin{prop}[the worst case]\label{lem_worst}
		The worst case happens if and only if  $ \gamma_1^\star =  \Vert \bm{\Lambda}^\star \Vert/  \Vert \bm{X}^\star \Vert $  and $ \gamma_2^\star = 1 $.
	\end{prop}
	\begin{proof}
	In view of \eqref{step_size2},  jointly optimizing $ \mathcal{M}_2 $ and $ \mathcal{M}_1 $ cannot be worse than optimizing $ \mathcal{M}_1 $  alone, 
	since otherwise one can always set $ \mathcal{M}_2 =   \mathcal{I}$  and $ \mathcal{M}_1  = ( \Vert \bm{\Lambda}^\star \Vert/  \Vert \bm{X}^\star \Vert ) \,\mathcal{I} $   to improve the performance, 
	which is a contradiction to  the solution being optimal. 
	
	The reverse  also holds due to the  uniqueness of the worst case. To see this, recall  from Prop. \ref{prop_blk} (strictly convex function), the joint optimal solution is always unique. This  implies that the above worst-case is also unique. 	
	The proof is now concluded.	
	\end{proof}

	\subsection{Worst-case  condition}	\label{sec_simple_data}
	Here,  we  study the necessary and  sufficient condition for the worst case. 	 We will see that the worst case arises if and only if the data is `simple-structured', i.e., no ill-conditioning structure within sub-blocks. 
	In this case,
	we can prove that 
	all $ (N-1) $  partition choices are equivalent,  leading to  the unique worst case.

	To start, we  need one lemma.
	\begin{lem}\label{basic}
		The following holds:
		\begin{equation}
		\frac{a}{b} = \frac{c}{d}  \,\,\Longleftrightarrow\,\,  \frac{a+c}{b+d} = \frac{a-c}{b-d}.
		\end{equation}
	\end{lem}
	\begin{proof}	
		Expand the terms. The left-hand-side  gives $ ad = bc $, and  the right-hand-side  gives $ bc-ad + (ab - dc) = ad - bc + (ab - dc) $,  which reduces to $ ad = bc $.
		The proof is therefore concluded.
	\end{proof}
	Now we are ready for the result. 
	\begin{prop}[worst-case condition]\label{fail_cond}
		Denote the $ i $th-column of  $ \bm{X}^\star, \bm{\Lambda}^\star $ as $ \bm{x}_{i}^\star, \bm{\lambda}_{i}^\star $, respectively.
		Suppose  $\exists\,  K \in  \{1, 2,\dots, N-1\} $ such that
		\begin{equation} \label{rel_ori}
		\frac{ \sum_{i=1}^K\Vert \bm{\lambda}_{i}^\star\Vert^2  }{ \sum_{i=1}^K\Vert \bm{x}_{i}^\star\Vert^2  } = 
		\frac{ \sum_{i=K+1}^{N}\Vert \bm{\lambda}_{i}^\star\Vert^2  }{\sum_{i=K+1}^{N} \Vert \bm{x}_{i}^\star\Vert^2 }. \tag{worst. cond.}
		\end{equation} 
		Then, 
		\begin{equation}
		 \gamma_1^\star =  \frac{\Vert \bm{\Lambda}^\star \Vert}{ \Vert \bm{X}^\star \Vert} , \quad  \gamma_2^\star = 1,
		\end{equation}
	 and  vice versa.
	 Moreover, the  above  worst case cannot be  avoided by employing a different partitioning.	
	\end{prop}
	\begin{proof}
		For comparison purpose, let us note that the column relation  \eqref{rel_ori} can be rewritten into a  block relation
		\begin{equation}\label{rel1}
		\frac{ \sum_{i=1}^K\Vert \bm{\lambda}_{i}^\star\Vert^2  }{ \sum_{i=1}^K\Vert \bm{x}_{i}^\star\Vert^2  } = 
		\frac{ \sum_{i=K+1}^{N}\Vert \bm{\lambda}_{i}^\star\Vert^2  }{\sum_{i=K+1}^{N} \Vert \bm{x}_{i}^\star\Vert^2 }
		\quad \Longleftrightarrow \quad	
		\frac{\Vert \bm{\Lambda}^\star_1 \Vert^2 + \Vert \bm{\Lambda}^\star_0 \Vert^2}{\Vert \bm{X}^\star_1 \Vert^2 + \Vert \bm{X}^\star_0 \Vert^2} =
		\frac{\Vert \bm{\Lambda}^\star_2 \Vert^2 + \Vert \bm{\Lambda}^\star_0 \Vert^2}{\Vert \bm{X}^\star_2\Vert^2 + \Vert \bm{X}^\star_0 \Vert^2}
		. 
		\end{equation}
		with sub-blocks $ \bm{X}_1,\bm{\Lambda}_1 \in \mathbb{S}^{K} $, $ \bm{X}_0,\bm{\Lambda}_0 \in \mathbb{R}^{K \times (N-K)} $,  and $ \bm{X}_2,\bm{\Lambda}_2 \in \mathbb{S}^{N-K} $.

		Recall from  \eqref{rho1_forms} that the following always holds:
		\begin{equation}\label{rel2}
	\frac{{\gamma_2^\star}^2\Vert \bm{\Lambda}^\star_1 \Vert^2 + {\gamma_2^\star}\Vert \bm{\Lambda}^\star_0 \Vert^2}{\Vert \bm{X}^\star_1 \Vert^2 + {\gamma_2^\star}\Vert \bm{X}^\star_0 \Vert^2}
	=\frac{\Vert \bm{\Lambda}^\star_2 \Vert^2 + {\gamma_2^\star}\Vert \bm{\Lambda}^\star_0 \Vert^2}{{\gamma_2^\star}^2\Vert \bm{X}^\star_2\Vert^2 + {\gamma_2^\star}\Vert \bm{X}^\star_0 \Vert^2} .
		\end{equation}
		Comparing the above \eqref{rel1} and \eqref{rel2}, we see that \eqref{rel1} is a special case when $\gamma_2^\star = 1$. In this case, we can invoke the $\gamma_1^\star $ expression in  \eqref{rho1_closed}, which gives
		\begin{equation}
			\gamma_1^\star  = \sqrt{\frac{ \Vert\bm{\Lambda}^\star_1 \Vert^2 + \Vert\bm{\Lambda}^\star_2 \Vert^2 + 2\Vert \bm{\Lambda}^\star_0 \Vert^2}{\Vert\bm{X}^\star_1 \Vert^2 + \Vert \bm{X}^\star_2\Vert^2 + 2\Vert \bm{X}^\star_0 \Vert^2}}
			= \frac{\Vert \bm{\Lambda}^\star \Vert}{ \Vert \bm{X}^\star \Vert} .
		\end{equation}
		Recall from  Prop. \ref{prop_blk} (strictly convex function) that the solution $ (\gamma_1^\star, \gamma_2^\star) $ is unique, the reverse therefore  also holds.


		Next, we show that  different partitioning choices coincide  with each other.
		To see this, we separate the $ K $-th column terms from \eqref{rel_ori}, yielding
		\begin{equation} 
		\frac{ \sum_{i=1}^{K-1}\Vert \bm{\lambda}_{i}^\star\Vert^2 + \Vert \bm{\lambda}_{K}^\star\Vert^2 }{ \sum_{i=1}^{K-1}\Vert \bm{x}_{i}^\star\Vert^2 + \Vert \bm{x}_{K}^\star\Vert^2 } =
		\frac{ \sum_{i=K}^{N}\Vert \bm{\lambda}_{i}^\star\Vert^2 - \Vert \bm{\lambda}_{K}^\star\Vert^2 }{\sum_{i=K}^{N} \Vert \bm{x}_{i}^\star\Vert^2 - \Vert \bm{x}_{K}^\star\Vert^2} 
		. 
		\end{equation}
		By Lemma \ref{basic}, the above implies that 
		\begin{equation} 
		\frac{ \sum_{i=1}^{K-1}\Vert \bm{\lambda}_{i}^\star\Vert^2  }{ \sum_{i=1}^{K-1}\Vert \bm{x}_{i}^\star\Vert^2 } = 
		\frac{ \sum_{i=K}^{N}\Vert \bm{\lambda}_{i}^\star\Vert^2 }{\sum_{i=K}^{N} \Vert \bm{x}_{i}^\star\Vert^2 } 
		,
		\end{equation}
		which corresponds to $ (K-1) $-type partitioning. 
		
		Similarly, if we separate the $ (K+1) $-th column terms from \eqref{rel_ori}, we arrive at
		\begin{equation}
		\frac{ \sum_{i=1}^{K+1}\Vert \bm{\lambda}_{i}^\star\Vert^2 - \Vert \bm{\lambda}_{K+1}^\star\Vert^2 }{ \sum_{i=1}^{K+1}\Vert \bm{x}_{i}^\star\Vert^2 - \Vert \bm{x}_{K+1}^\star\Vert^2 } =
		\frac{ \sum_{i=K+2}^{N}\Vert \bm{\lambda}_{i}^\star\Vert^2 + \Vert \bm{\lambda}_{K+1}^\star\Vert^2 }{\sum_{i=K+2}^{N} \Vert \bm{x}_{i}^\star\Vert^2 + \Vert \bm{x}_{K+1}^\star\Vert^2} 
		. 
		\end{equation}
		By Lemma \ref{basic}, the above implies that 
		\begin{equation}
		\frac{ \sum_{i=1}^{K+1}\Vert \bm{\lambda}_{i}^\star\Vert^2 }{ \sum_{i=1}^{K+1}\Vert \bm{x}_{i}^\star\Vert^2 } = 
		\frac{ \sum_{i=K+2}^{N}\Vert \bm{\lambda}_{i}^\star\Vert^2 }{\sum_{i=K+2}^{N} \Vert \bm{x}_{i}^\star\Vert^2 } 
		. 
		\end{equation}
		which corresponds to    $ (K+1) $-type partitioning. 
		Straightforwardly, one can continue this process and show the above relation  holds for all different partitioning approaches. 	
		The proof is therefore concluded.	
	\end{proof}

	\section{Application to QCQP}
	 Quadratically constrained quadratic program (QCQP)  is intrinsically connected to SDP  and has important uses across many fields,  especially in wireless communications.  We will present a  unified  primal-dual formulation, which clarifies some underlying structures. Particularly, we see that this formulation is slightly more general  than  the standard SDP, and	 
	 the primal problem is generally solvable,  while the dual does  not. 
	 Moreover,  we implement a closed-form metric-ADMM solver based on the primal formulation.

	\subsection{Non-convex QCQP}

	In this section,   we investigate the quadratically constrained quadratic programs (QCQPs). 
	Following the literature, we consider the standard form
	\begin{align}\label{qcqp_primal}
\qquad\qquad\qquad\qquad\quad	\underset{\bm{x}}{\text{minimize}}&\quad  \bm{x}^T\bm{A}_0 \bm{x} + 2\bm{b}_0^T\bm{x} \nonumber\\
\qquad\qquad\qquad\qquad\quad	\text{subject\,to}
	& \quad  \bm{x}^T\bm{A}_i \bm{x} + 2\bm{b}_i^T\bm{x} + c_i \leq 0, \quad  i = 1, \cdots, m. \tag{QCQP}
	\end{align}
	with   $  \bm{x} \in  \mathbb{R}^N $,  $  \bm{b} \in  \mathbb{R}^N $, and $ \bm{A}_i \in  \mathbb{S}^N, \ \forall i $  being symmetric matrices.
	The above  can be  equivalently cast  into  
	\begin{align}\label{ab_qcqp}
	\qquad\qquad\,\,  \underset{\bm{X}}{\text{minimize}}\quad\,\,\,  &   \langle\bm{E}_0, \bm{X}\rangle  \nonumber\\
	\qquad\qquad\,\,	\text{subject\,to}\quad\,\,\,				 &  \langle\bm{E}_i, \bm{X}\rangle \leq 0, \quad  i = 1,\dots, m ,\nonumber\\
																	 &  \text{rank} (\bm X_1) =  1.
	\end{align}
	with 
	\begin{equation}\label{var_def}
	\qquad
	\bm{X}        \eqdef  \left[\begin{array}{cc} \bm X_1 & \bm{x}\\ \bm{x}^T  & 1 \end{array} \right], \quad 
	\bm{E}_i      \eqdef  \left[\begin{array}{cc} \bm{A}_i   & \bm{b}_i\\ \bm{b}_i^T  & \bm{c}_i \end{array} \right],\,\, i = 0, 1,\dots, m.
	\end{equation}
	The above formulation is ill-posed/intractable,  due to  the non-convex rank-1 constraint. 
	We will   our discussion to a well-known convexified program, known as  semidefinite  relaxation.

	\subsection{Convex relaxation}
	Consider the Fenchel dual of \eqref{ab_qcqp}, which can be explicitly written as	
		\begin{align}\label{pri0}
	\qquad\qquad	\underset{\bm{x},  x_{22}}{\text{minimize}}&\quad  - \langle\bm{x},\bm{c}\rangle - x_{22}  \nonumber\\
	\qquad\qquad	\text{subject\,to}
	& \quad \left[\begin{array}{cc} \bm{A}_0 + \sum_{i = 1 }^{m} x_i\bm{A}_i & \bm{b}_0 + \bm{B}\bm{x}\\ (\bm{b}_0 + \bm{B}\bm{x})^T  & -x_{22} \end{array} \right] \succeq 0, \nonumber\\
	& \quad \bm{x} \geq 0,  \tag{primal}
	\end{align}
	where
	\begin{equation}
	\bm{B}\eqdef [\bm{b}_1\cdots, \bm{b}_m] \in \mathbb{R}^{N \times m} , \qquad
	 \bm{c}\eqdef [{c}_1\cdots, {c}_m] \in \mathbb{R}^{m }.
	\end{equation}
	By appealing to the duality,  the above program is now convex. Following  the literature of SDP, we  refer to the above as the primal problem. 
	Take  the dual operation again, we  obtain the dual problem as
	\begin{align}\label{dual0}
	\qquad\qquad\,\,  \underset{\bm{X}}{\text{minimize}}\quad\,\,\,  &   \langle\bm{E}_0, \bm{X}\rangle  \nonumber\\
	\qquad\qquad\,\,	\text{subject\,to}\quad\,\,\,				 &  \langle\bm{E}_i, \bm{X}\rangle \leq 0, \quad  i = 1,\dots, m ,\nonumber\\
	&  \bm{X}  \succeq 0,  \tag{dual}
	\end{align}
	with variables defined in \eqref{var_def}, restated here as  
	\begin{equation}
	\qquad
	\bm{X}        \eqdef  \left[\begin{array}{cc} \bm X_1 & \bm{x}\\ \bm{x}^T  & 1 \end{array} \right], \quad 
	\bm{E}_i      \eqdef  \left[\begin{array}{cc} \bm{A}_i   & \bm{b}_i\\ \bm{b}_i^T  & \bm{c}_i \end{array} \right],\,\, i = 0, 1,\dots, m.
	\end{equation}

	Compare the above to the original QCQP,  we see that the only difference  is that the non-convex rank-1 constraint is replaced by a convex positive semidefinite requirement. In view of  this, such a approach is  often referred to as the \textit{semidefinite relaxation}. Particularly, let us note that
	\begin{equation}\label{SDR}
\bm{X} \eqdef \,\,  \left[\begin{array}{cc} \bm{X}_1& \bm{x}\\ \bm{x}^T  & 1 \end{array} \right] \succeq 0  \quad \Longleftrightarrow \quad	\bm{X}_1\succeq \bm{x}\bm{x}^T ,
	\end{equation}
	with equality attained if and only if $ \bm{X}_1 = \bm{x}\bm{x}^T  $, i.e., $ \text{rank} (\bm X_1) =  1 $. In which case, the relaxation is  tight.

	\subsubsection{Unified treatment}
	For  the ADMM solver, there is  no iteration-number-complexity difference for  solving \eqref{pri0}  and \eqref{dual0}. One may simply select the problem that is  easier to solve. In fact, no matter which problem one eventually chooses, the ADMM iterates share the same structure.
\begin{prop}[unified treatment]\label{prop_uni}
	The primal-dual problem  \eqref{pri0}  and \eqref{dual0} can be solved in the following unified way:
	\begin{align}\label{sdp_uni0}
&\,\underset{\bm{X},\bm{Z}}{\text{minimize}}\quad   f\,(\bm{X}) +  \delta_{\mathbb{S}^{N+1}_+} (\bm{Z})   \qquad \nonumber\\
&\text{subject\,to}\quad  \mathcal{A}\,  \big(\bm{X}\big) = \bm{Z}, 					
\end{align}
via  ADMM iterates:
\begin{align}\label{admm_sca0}
\bm{X}^{k+1} =\,\,& \underset{\bm{X}}{\text{argmin}} \,\,\,	f(\bm{X})  + \frac{1}{2}\Vert\mathcal{S}\mathcal{A}\bm{X} - \widetilde{\bm Z}^k  + \widetilde{\bm \Lambda}^k\Vert^2  \nonumber\\				
\widetilde{\bm Z}^{k+1} =\,\, &   	\Pi_{\mathbb{S}^{N+1}_+}\, \big(\, \widetilde{\bm{X}}^{k+1} + \widetilde{\bm{\Lambda}}^{k}  \big) \nonumber\\
\widetilde{\bm\Lambda}^{k+1} =\,\,  & \widetilde{\bm\Lambda}^{k} +  \widetilde{\bm{X}}^{k+1} - \widetilde{\bm  Z}^{k+1} ,	   \tag{scaled metric-ADMM}
\end{align}	
with  variable  substitutions:
\begin{equation}\label{sc_var}
\widetilde{\bm{X}}^k =  \mathcal{S}\mathcal{A}\bm{X}^{k},  \qquad
\widetilde{\bm Z}^{k} =  \mathcal{S}\bm{Z}^{k}, \qquad \widetilde{\bm{\Lambda}}^{k+1} = (\mathcal{S}^*)^{-1}\bm{\Lambda}^{k}.
\end{equation}

For \eqref{pri0}, the matrix variable $ \bm{X} $ reduces to a vector $ [\bm{x},  x_{22}]^T	 $, and
the following definitions are used: 
\begin{align}\label{pri_d0}
\quad(\text{Function}) \qquad\,\,\,	&\,f\,\big([\bm{x},  x_{22}]^T \big) \,= \,\, f_1(\bm{x})  - x_{22} =  - \langle\bm{x},\bm{c}\rangle - x_{22}	, \quad \text{dom}\, f_1 =  \mathscr R_{+}, 	\nonumber\\
\,\,(\text{Constraint} )  \qquad 	 & \mathcal{A}\,  \big([\bm{x},  x_{22}]^T \big) =  \bm{A}_0 + \sum_{i=1}^m  x_i\bm{A}_i +  x_{22} \big( \bm{b}_0 + \bm{B}\bm{x} \big)\big( \bm{b}_0 + \bm{B}\bm{x} \big)^T.
\end{align}

For \eqref{dual0}, the following definitions are used: 
\begin{align}
\quad(\text{Function}) \qquad\,\,\,	&f\,(\bm{X}) \,\,=\langle\bm{E}_0, \bm{X}\rangle , \quad  \text{dom}\, f =  \{ \bm{X} \in \mathbb{S}^{N+1} \,\, |\,\,  \langle\bm{E}_i, \bm{X}\rangle \leq 0,  \,\, \forall i  \} 		\nonumber\\
\,\,(\text{Constraint} ) \qquad	 &\mathcal{A}\,  \big(\bm{X}\big) =  \bm{X}.
\end{align}
	\end{prop}

	\subsubsection{A  closed-form  solver}\label{sec_pri_solver}
	Here,   we  specify  the closed-form ADMM iterates for \eqref{pri0}, due  to its general solvability.  For the dual problem,  the challenge lies on the inequalities $  \langle\bm{E}_i, \bm{X}\rangle \leq 0 , \forall  i$,  which appear  not convenient to handle in  general.


	
	In view of ADMM \eqref{admm_sca0}, only the $ \bm{X} $-update step is implicit. Hence, our goal is to find the explicit solution  of the  following  problem:
	\begin{equation}
	\underset{\bm{X}}{\text{argmin}} \,\,\,	f(\bm{X})  + \frac{1}{2}\Vert\mathcal{S}\mathcal{A}\bm{X} - \widetilde{\bm Z}^k  + \widetilde{\bm \Lambda}^k\Vert^2.
	\end{equation}
	
	To  this  end,  invoke the primal problem definitions in   \eqref{pri_d0}, the above is  specified into 
	\begin{equation}
	\underset{\bm{x}\in \mathscr R_{+},  x_{22}}{\text{argmin}} \,\,\, - \langle\bm{x},\bm{c}\rangle - x_{22} 
	+ \frac{1}{2} \left\Vert 
	\left[\begin{array}{cc} \bm{S}_1 &      \bm{s}_0   \\ \bm{s}_0^T &    {s}_{22}  \end{array}\right] \, \odot\, 
	\left[\begin{array}{cc} \bm{A}_0 + \sum_i x_i\bm{A}_i & \bm{b}_0 + \bm{B}\bm{x}\\ (\bm{b}_0 + \bm{B}\bm{x})^T  & -x_{22} \end{array} \right] - \widetilde{\bm Z}^k  + \widetilde{\bm \Lambda}^k  \right\Vert^2,
	\end{equation}
	with $ \bm{S}_1  \in \mathbb{S}^N $, $ \bm{s}_0  \in \mathbb{R}^{N\times 1 } $, and  $ {s}_{22}  \in \mathbb{R} $.	
	Clearly, this problem is separable w.r.t.  $  \bm{x} $  and $ x_{22} $.

	To proceed, define the following  partitions for the  scaled  variables  as in \eqref{sc_var}:
	\begin{equation}\label{par}
	\widetilde{\bm Z} = \left[\begin{array}{cc} \widetilde{\bm Z}_1 & \widetilde{\bm z}_0\\ \widetilde{\bm z}_0^\text{T} & z_{22} \end{array}\right],  \quad
	 \widetilde{\bm \Lambda} = \left[\begin{array}{cc}  \widetilde{\bm \Lambda}_1 & \widetilde{\bm \lambda}_0\\  \widetilde{\bm \lambda}_0^\text{T} & \lambda_{22} \end{array}\right],
	\end{equation}
	with  $ \widetilde{\bm Z}, \widetilde{\bm \Lambda}  \in \mathbb{S}^N $, $ \widetilde{\bm z}_0, \widetilde{\bm \lambda}_0  \in \mathbb{R}^{N\times 1 } $, and  $ z_{22} ,  \lambda_{22} \in \mathbb{R} $.	
	We arrive at
	\begin{equation}
	\underset{\bm{x}\in \mathscr R_{+}}{\text{argmin}} \,\, - \langle\bm{x},\bm{c}\rangle
	+ \frac{1}{2} \left\Vert \bm{S}_1  \, \odot\, \big( \bm{A}_0 + \sum_i x_i\bm{A}_i \big) - \widetilde{\bm Z}_1^k  + \widetilde{\bm \Lambda}_1^k  \right\Vert^2
	+  \left\Vert \bm{s}_0  \, \odot\, \big( \bm{b}_0 + \bm{B}\bm{x} \big) - \widetilde{\bm z}_0^k  + \widetilde{\bm \lambda}_0^k  \right\Vert^2,
	\end{equation}
	which  can be rewritten  into
	\begin{equation}
	\underset{\bm{x}}{\text{argmin}} \,\, - \langle\bm{x},\bm{c}\rangle
	+ \frac{1}{2} \left\Vert  \text{vec}( \bm{S}_1 ) \odot \widetilde{\bm{A}}\bm x + \text{vec} \big(  \bm{S}_1  \odot  \bm{A}_0 - \widetilde{\bm Z}_1^k  + \widetilde{\bm \Lambda}_1^k  \big) \right\Vert^2
	+  \left\Vert \bm{s}_0  \, \odot\, \big( \bm{b}_0 + \bm{B}\bm{x} \big) - \widetilde{\bm z}_0^k  + \widetilde{\bm \lambda}_0^k  \right\Vert^2, \,\, \bm{x}\geq 0,
	\end{equation}	
	where
	\begin{equation}\label{A_t}
	 \widetilde{\bm{A}} \eqdef [ \text{vec}(\bm{A}_1) , \cdots, \text{vec}(\bm{A}_m)  ],
	\end{equation}
	and  where $ \text{vec} (\cdot)$  denotes a vectorization step.
	The  first-order optimality condition  is therefore given by
	\begin{align}
	0 =&  - \bm{c}	+   \widetilde{\bm{A}}^T  \text{vec}( \bm{S}_1 ) \odot\bigg(  \text{vec}( \bm{S}_1 ) \odot \widetilde{\bm{A}}   \bm x + \text{vec} \big(  \bm{S}_1  \odot  \bm{A}_0 -\widetilde{\bm Z}_1^k  +\widetilde{\bm \Lambda}_1^k  \big)   \bigg) \nonumber\\
	&+  2 \bm{B}^T \bm{s}_0  \odot \bigg(   \bm{s}_0  \odot   \bm{B}\bm{x} +   \bm{s}_0   \odot \bm{b}_0 -   \widetilde{\bm z}_0^k  +   \widetilde{\bm \lambda}_0^k  \bigg) , \quad \bm{x}\geq 0,
	\end{align}
	which gives solution
	\begin{equation}
	\bm x = \bm{D}^{-1} (\bm{t}_1  + \bm{t}_2  +  \bm{c}), \quad \bm{x}\geq 0,
	\end{equation}
	with 
	\begin{align}\label{v_def}
	&\bm{D}  	\,\,\eqdef\quad    \widetilde{\bm{A}}^T  \text{vec}( \bm{S}_1 ) \odot  \text{vec}( \bm{S}_1 ) \odot \widetilde{\bm{A}} +   2 \bm{B}^T  \bm{s}_0  \odot   \bm{s}_0  \odot   \bm{B},   \nonumber\\
	&\bm{t}_1   \,\,\eqdef\,\,   -  \widetilde{\bm{A}}^T  \text{vec}( \bm{S}_1 ) \odot  \text{vec} \bigg(  \bm{S}_1  \odot  \bm{A}_0 -\widetilde{\bm Z}_1^k  +\widetilde{\bm \Lambda}_1^k  \bigg)  ,\nonumber\\
	&\bm{t}_2 	\,\,\eqdef\,\,   -  2 \bm{B}^T \bm{s}_0  \odot \bigg(   \bm{s}_0  \odot   \bm{B}\bm{x} +   \bm{s}_0   \odot \bm{b}_0 -   \widetilde{\bm z}_0^k  +   \widetilde{\bm \lambda}_0^k  \bigg) .
	\end{align}
	At last, the non-negativity requirement  is easy to satisfy, and we arrive the final result:
	\begin{equation}\label{x_s}
	\bm x^{k+1} = \text{max} \bigg\{ \, \bm{D}^{-1} (\bm{t}_1  + \bm{t}_2  +  \bm{c}), \,\,  \bm{0}  \, \bigg\}.
	\end{equation}

		For scalar variable $ x_{22} $,  the answer  is straightforward
	\begin{equation}
	\underset{x_{22} }{\text{argmin}} \,\, - x_{22} 
	+ \frac{1}{2} \left\Vert - {s}_{22} \cdot x_{22} - \widetilde{ z}_{22}^k  + \widetilde{\lambda}_{22}^k  \right\Vert^2,
	\end{equation}
	with
	\begin{equation}
	x_{22}^{k+1}  =  \frac{1}{{s}_{22}} \bigg(  - \widetilde{z}_{22}^k  + \widetilde{\lambda}_{22}^k +  \frac{1}{{s}_{22}}  \bigg).
	\end{equation}

	In summary, we obtain the following metric ADMM  solver.
	\begin{algorithm}[H]
		\caption{  Metric-ADMM solver  for convexified QCQP}
		\begin{algorithmic}[1]\label{qcqp_solver}
			\STATE Let  $\bm{\Lambda}^{0}\leftarrow\bm{0}$,\,\, $\bm{Z}^{0}\leftarrow\bm{0}$.  Recall  partitioning  \eqref{par}.
			\STATE{Repeat the following iterations until convergence:}
			\vspace{5pt}
			\STATE \quad\quad $ 	\bm x^{k+1} \leftarrow \text{max} \bigg\{ \, \bm{D}^{-1} (\bm{t}_1  + \bm{t}_2  +  \bm{c}), \,\,  \bm{0}  \, \bigg\}  $,   \quad // see definition \eqref{v_def}\vspace{8pt} 
			\STATE \quad\quad $ 	x_{22}^{k+1} \leftarrow \frac{1}{{s}_{22}} \big(  - \widetilde{z}_{22}^k  + \widetilde{\lambda}_{22}^k +  \frac{1}{{s}_{22}} \big)    $,     \vspace{8pt}  
			\STATE \quad\quad $ \bm{X}^{k+1}  \eqdef \left[\begin{array}{cc} \bm{A}_0 + \sum_{i = 1 }^{m} x_i\bm{A}_i & \bm{b}_0 + \bm{B}\bm{x}\\ (\bm{b}_0 + \bm{B}\bm{x})^T  & -x_{22} \end{array} \right]  $,  \vspace{8pt} 
			\STATE \quad\quad$ \widetilde{\bm{Z}}^{k+1}  \leftarrow 	\Pi_{\mathbb{S}^{N+1}_+}\, \big( \mathcal{S}\bm{X}^{k+1} + \widetilde{\bm{\Lambda}}^{k}  \big)     $,    \vspace{8pt}
			\STATE \quad\quad$ \widetilde{\bm{\Lambda}}^{k+1} \leftarrow  \widetilde{\bm\Lambda}^{k} +   \mathcal{S}\bm{X}^{k+1} - \widetilde{\bm  Z}^{k+1}    $.		\vspace{5pt}	
		\end{algorithmic}
	\end{algorithm}

	\begin{remark}[matrix inverse]
	In view of \eqref{x_s}, we note that the solution involves computing a matrix inverse. At first glance, it may be expensive. Let us note  that $ \bm{D} \in \mathbb{R}^{m\times m} $, where $ m $ is typically  small. 
	Moreover, for interpretation purpose, we may view \eqref{dual0}, where $ m $ corresponds to the number of constraints. Clearly, as  $ m $ increases, the problem becomes more complicated. Not surprisingly, 
	we should expect the algorithm runtime to increase, and this is reflected as the increased difficulty to calculate   $ \bm{D}^{-1} $.
	\end{remark}

	\subsection{Reduction to standard SDP}
Here,  we show that by simply removing some  terms in Algorithm \ref{qcqp_solver}, we obtain a metric-ADMM  solver for the standard SDP problems.  In  this sense,  the QCQP is slightly more general than the standard   SDP. 
	
  Consider  the case with  $ \bm{b}_i= \bm{0},\,\,i=0,1,\dots,m $.  \eqref{pri0} can be rewritten into
		\begin{align}
	\qquad\qquad	\underset{\bm{x},  x_{22}}{\text{minimize}}&\quad  - \langle\bm{x},\bm{c}\rangle - x_{22}  \nonumber\\
	\qquad\qquad	\text{subject\,to}
	& \quad \left[\begin{array}{cc} \bm{A}_0 + \sum_{i = 1 }^{m} x_i\bm{A}_i & \bm{0}\\ \bm{0} & -x_{22} \end{array} \right] \succeq 0, \nonumber\\
	& \quad \bm{x} \geq 0,
	\end{align}
	In this case,  we instantly have the solution $ x_{22}^\star =0 $, since $ -x_{22}  \geq 0 $  (by  Schur complement  lemma) and  we are minimizing $ -x_{22} $.
	That  said,  the above problem is equivalent to
	\begin{align}
	\qquad\qquad	\underset{\bm{x}}{\text{minimize}}&\quad \langle\bm{x},  - \bm{c}\rangle  \nonumber\\
	\qquad\qquad	\text{subject\,to}
	& \quad  \bm{A}_0 + \sum_{i = 1 }^{m} x_i\bm{A}_i  \succeq 0, \nonumber\\
	& \quad \bm{x} \geq 0,
	\end{align}
	Furthermore, if we  also remove the positivity constraint  $ \bm{x} \geq 0 $,   then  we obtain exactly the  standard SDP problem. 
Hence,  in a simple  way, we derived the   metric-ADMM solver for the standard SDP.
		\begin{algorithm}[H]
		\caption{  Metric-ADMM solver  for the standard  SDP}
		\begin{algorithmic}[1]
			\STATE Let  $\bm{\Lambda}^{0}\leftarrow\bm{0}$,\,\, $\bm{Z}^{0}\leftarrow\bm{0}$.  Recall  partitioning  \eqref{par}.
			\STATE{Repeat the following iterations until convergence:}
			\vspace{5pt}
			\STATE \quad\quad $ 	\bm x^{k+1} \leftarrow  \bm{D}^{-1} (\bm{t}_1   +  \bm{c})  $,   \quad // see definition \eqref{v_def} with $ \bm{B} = \bm{0} $\vspace{8pt} 
			\STATE \quad\quad $ \bm{X}^{k+1}  \eqdef  \bm{A}_0 + \sum_{i = 1 }^{m} x_i\bm{A}_i $,  \vspace{8pt} 
			\STATE \quad\quad$ \widetilde{\bm{Z}}^{k+1}  \leftarrow 	\Pi_{\mathbb{S}^{N}_+}\, \big( \mathcal{S}\bm{X}^{k+1} + \widetilde{\bm{\Lambda}}^{k}  \big)     $,    \vspace{8pt}
			\STATE \quad\quad$ \widetilde{\bm{\Lambda}}^{k+1} \leftarrow  \widetilde{\bm\Lambda}^{k} +   \mathcal{S}\bm{X}^{k+1} - \widetilde{\bm  Z}^{k+1}    $.		\vspace{5pt}	
		\end{algorithmic}
	\end{algorithm}

%
%
%

	\subsection{Specific applications}\label{sec_app}
	In this section, we briefly present  two  specific applications that admit the  QCQP structure.

	\subsubsection{Matrix-fractional problem}
First, consider the matrix-fractional problems \cite[Sec. 2.4]{lobo1998applications}
	\begin{align}
	\underset{\bm{x}}{\text{minimize}}&\quad  (\bm{b}_0 + \bm{B}\bm{x})^T( \bm{A}_0 + \sum_i x_i\bm{A}_i)^{-1}(\bm{b}_0 + \bm{B}\bm{x})  \nonumber\\
	\text{subject\,to}
	& \quad \bm{A}_0 + \sum_i x_i\bm{A}_i \succeq 0,\nonumber\\
	& \quad \bm{x} \geq 0.
	\end{align}
	By Schur compliment lemma, the above can be rewritten into
	\begin{align}\label{mat_frac_formu}
	\underset{\bm{x}, t}{\text{minimize}}&\qquad   t \nonumber\\
	\text{subject\,to}
	& \quad \left[\begin{array}{cc} \bm{A}_0 + \sum_i x_i\bm{A}_i & \bm{b}_0 + \bm{B}\bm{x}\\ (\bm{b}_0 + \bm{B}\bm{x})^T  & t \end{array} \right] \succeq 0,\nonumber\\
	& \quad \bm{x} \geq 0,   \tag{mat. fra.}
	\end{align}
	Comparing the above to \eqref{pri0}, we see that it is a special case of $ \bm{c} = \bm{0} $ and variable $ t = - x_{22} $.

	\subsubsection{Boolean quadratic program}	
	Here, we consider the Boolean quadratic program (BQP), which is a fundamental problem in digital communication. 
		
	The original non-convex  form is given by
	\begin{align}
	\underset{\bm{x}}{\text{minimize}}&\quad   \bm{x}^T\bm{A}_0\bm{x} + 2\bm{b}_0^T\bm{x} \nonumber\\
	\text{subject\,to}
	& \quad {x}_i^2 = 1, \quad \forall i.
	\end{align}
	Its semidefinite relaxation is well-known to be 
	\begin{align}\label{r_BQP}
	\underset{\bm{X}}{\text{minimize}}\quad  & \langle  \bm{E}_0, \bm{X} \rangle \nonumber\\
	\text{subject\,to}\quad 
		& \text{diag}(\bm{X})= \bm{1}\nonumber \\ 
	& \bm{X} \succeq 0 , \tag{BQP dual}
	\end{align}
	with
	\begin{equation}
	\bm{E}_0	\eqdef  \left[\begin{array}{cc} \bm{A}_0\,\,\, & \bm{b}_0 \\ \bm{b}_0^\text{T}  & 0 \end{array} \right], \quad \bm{X}	\eqdef  \left[\begin{array}{cc} \bm{X}_1\,\,\, & \bm{x} \\ \bm{x}^\text{T}  & 1 \end{array} \right].
	\end{equation}
	Comparing the above BQP problem to \eqref{dual0},  
	it can be viewed a special case of  choosing (i)  $ \bm{E}_i  $ to be elementary matrix with only the $ i $-th diagonal element equals $ 1 $ and zeros elsewhere; (ii)  $ \bm{c} =  -\bm{1} $, i.e., a negative  ones  vector.

	 It follows  that,  the primal  formulation of  the BQP is  
	\begin{align}\label{bqp_pri}
	\qquad\qquad	\underset{\bm{x},  x_{22}}{\text{minimize}}&\quad   \langle\bm{x},\bm{1} \rangle - x_{22}  \nonumber\\
	\qquad\qquad	\text{subject\,to}
	& \quad \left[\begin{array}{cc} \bm{A}_0 + \sum_{i = 1 }^{N} x_i\bm{A}_i & \bm{b}_0 \\ \bm{b}_0^T  & -x_{22} \end{array} \right] \succeq 0, \nonumber\\
	& \quad \bm{x} \geq 0,  \tag{BQP  primal}
	\end{align}
	At first  glance, solving the above primal formulation is less efficient (runtime sense) than \eqref{r_BQP}, because we  have $ m=N $  number  of matrices $ \bm{A}_i $.  However, we  emphasize that $ \bm{A}_i $  is highly structured, and essentially  the LMI is simply
	\begin{equation}
	\bm{A}_0 + \text{Diag}(\bm{x}),  
	\end{equation}
	where $ \text{Diag}(\cdot) $ denotes reshaping  a vector into a diagonal  matrix. 
	Moreover, in view of Algorithm \ref{qcqp_solver},  the matrix inversion  $ \bm{D}^{-1} $ reduces to an element-wise division operation. We believe it  is completely the same for solving   the  BQP primal  and dual problems.

	\section{Numerical results}
	In this section, we test the  numerical  performance of our metric-ADMM  solver  via  the two applications  in Sec. \ref{sec_app}.
	Recall the  solver implementation in Algorithm \ref{qcqp_solver}, and optimal metric choice determined via Algorithm \ref{Algo_1}. 
	The stopping criteria  is based on $ (1/N) \Vert \bm{x}^k - \bm x^\star \Vert^2 \leq \epsilon$, with $ \bm x^\star \in \mathbb{R}^N $ denotes the optimal solution, where $ \epsilon $ is set to $ 1\times 10^{-8} $.\\

	\noindent\textbf{General data setting:}   Unless specified, all  data  is  generated randomly  via  zero-mean normal distribution  $ \mathcal{N}(0,\sigma) $. 	For reproduction purpose, we fix the random number generator to  the `default' mode in MATLAB
	for all simulations. 	
	Particularly, we guarantee the positive (semi)  definiteness of a matrix $ \bm{A}  $ by operation  $ \sqrt{\bm{A}^T \bm{A}} $. 
	\\

	\noindent\textbf{Parameter settings:}  
	We will investigate 3 types  of parameters: 
	(i) the `scalar limit', referring to the underlying best scalar step-size,  found by exhaustive search;
	(ii) $ \mathcal{M}^\star $, denoting the optimal choice of the variable metric;	 
	(iii) $ \gamma^\star =  \Vert \bm{\Lambda}^\star \Vert / \Vert \mathcal{A}\bm{X}^\star\Vert $.
	\\

	\noindent\textbf{The worst case:} 
	As shown in Sec. \ref{sec_simple_data},  the  worst case arises if the data  is simple-structured,  i.e., no ill-conditioning  structure within the  block variable. 
	In the simulations, we find that if we randomly generate all data in the exact  same  way,  then the worse  case  arises,  i.e., 
	condition \eqref{rel_ori}  roughly holds.  Indeed, numerical results show that there is hardly  any   efficiency  gain in  this  case.


	\subsection{Matrix-fractional problem}
In this section, we consider the matrix-fractional problem introduced in \eqref{mat_frac_formu}.
The problem data is fully characterized by  $ \{\bm{A}_i, \bm{b}_i\}_{i=0}^m $.

\subsubsection{Scalability: different conditionings}
Here,  we evaluate scalability  property in 
the iteration number complexity sense.  The elements of matrices $ \{\bm{A}_i\}_{i=0}^m $ and vectors $ \{\bm{b}_i\}_{i=0}^m $ are randomly generated from  $  \mathcal{N}(0,\sigma_A)  $ and $  \mathcal{N}(0,\sigma_b)  $, respectively.
 
	\begin{figure}[H]
		\centering
		\begin{subfigure}{0.32\textwidth}
			\includegraphics[scale=0.32]{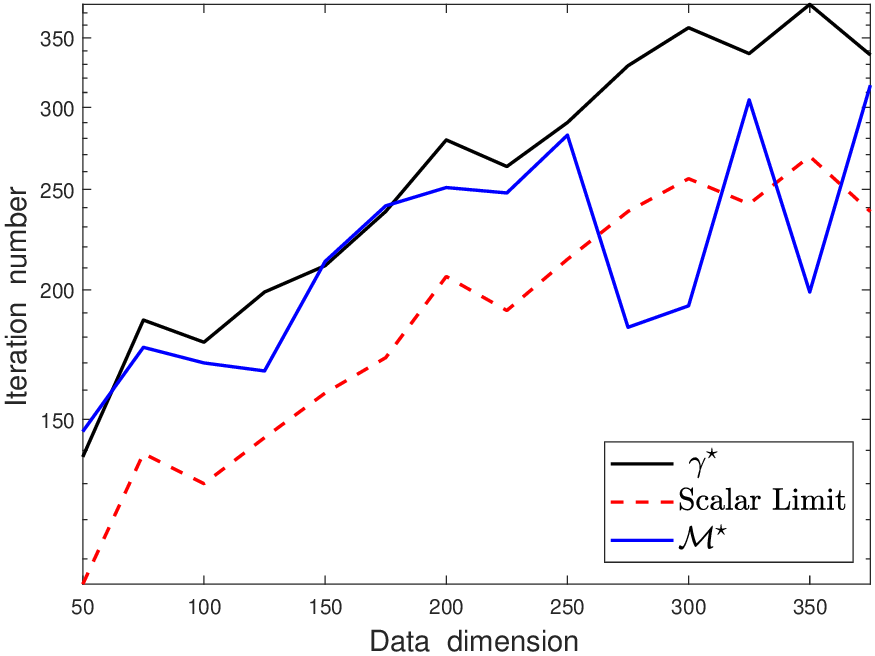}\caption{Worst case: $ \sigma_A =1,\,\,  \sigma_b= 1,\,\, 	 m = 5 $.}
		\label{a}
	\end{subfigure}
		\begin{subfigure}{0.32\textwidth}
			\includegraphics[scale=0.32]{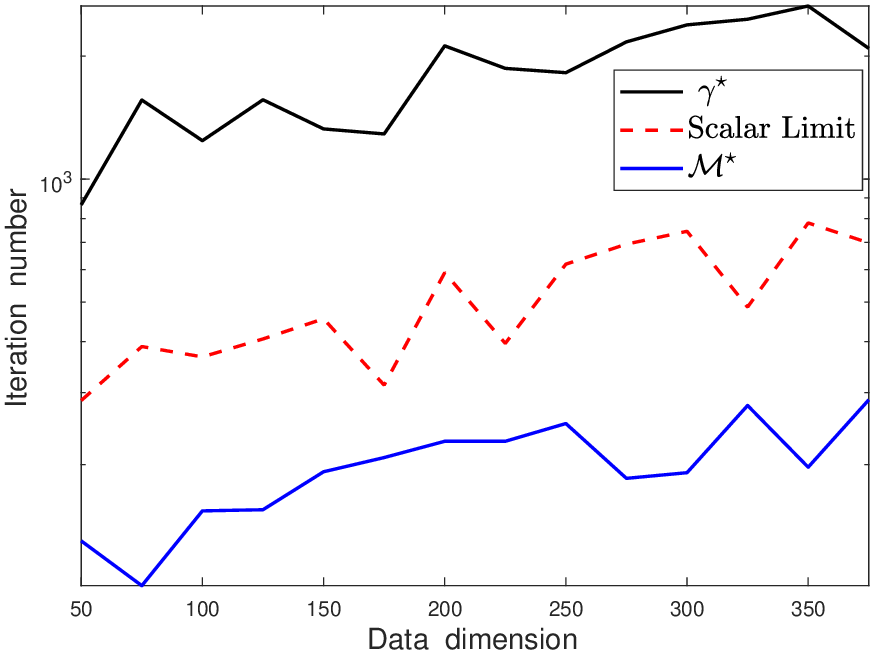}\caption{$ \sigma_A =2,\,\,  \sigma_b= 1/2,\,\, 	 m = 5 $.}
		\label{b}
	\end{subfigure}
			\begin{subfigure}{0.32\textwidth}
			\includegraphics[scale=0.32]{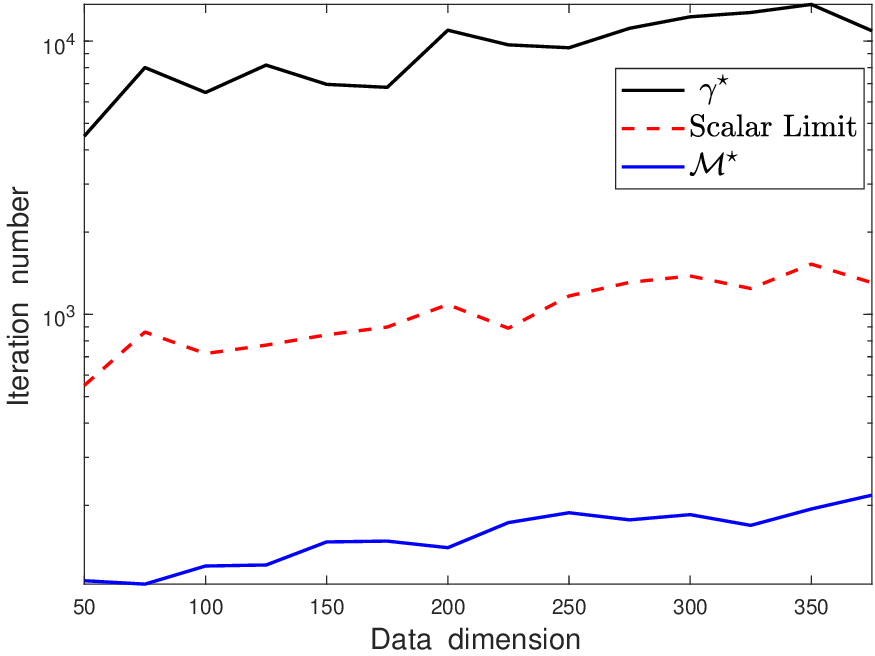}\caption{ $ \sigma_A = 3,\,\,  \sigma_b= 1/3,\,\,   m = 5 $.}
		\label{c}
	\end{subfigure}
		\caption{Iteration number complexity against data dimension.}
	\end{figure}
	
	\noindent\textbf{Performance:} 
	In Fig. \ref{a}, the  data is generated  in the exact same way.  In this  case, we  find that a scalar parameter would (almost) fully exploit the  ill-conditioning  structure, corresponding  to the worst case  of  our metric-ADMM. 
	In Fig. \ref{b}, by generating the two  types of data $ \{\bm{A}_i\}_{i=0}^m $ and  $ \{\bm{b}_i\}_{i=0}^m $  in a `reverse' way, i.e., $ \sigma_A = 1/  \sigma_b $, we promote a  special ill-conditioning  structure that cannot   be  fully exploited by a scalar parameter.  In  this case, our metric-ADMM shows roughly one order of  magnitude advantage. In Fig. \ref{c},  we  promote  a  stronger ill-conditioning structure, and  the advantage increases to roughly two  orders of  magnitudes.  
	When we further increase the ill-conditioning,  the advantage also further increases (the extra  experiments are  omitted due to limited space).  
	In all cases, we  observe that the metric-ADMM has similar performances, with iteration number roughly at $ 200 $, and increases  very slowly as the dimension increases, implying a scalable  solver is obtained.


%
%
%
%
%
%
%

	\subsection{Boolean QP}
	In this section, we consider the Boolean quadratic program problem.
	The problem data is fully characterized by $\bm{A}_0$ and  $\bm{b}_0$.

	\subsubsection{Natural ill-conditioning}
For BQP,  we note that there exists a natural ill-conditioning structure  (that cannot be fully exploited via a  scalar parameter),  appears related to requiring  the  underlying  solution being integers.

	\noindent\textbf{Data setting:}  We generate the BQP data in the  following manner: (i)  Randomly generate an integer vector $ \bm x_0 $ with entries being either $ 1 $ or $ -1 $. To achieve this, we  first randomly generate its elements via $ \mathcal{N}(-0.5, 1) $  and then take the sign operation  element-wisely (restart if a zero element arises). (ii)    Randomly generate elements of matrix $\bm{A}_0$ via  $  \mathcal{N}(0, 1)  $. (iii) Set	$ \bm b_0 = \bm{A}_0 \bm{x}_0 +  \epsilon $, with noise $ \epsilon $  randomly generate  via $ \mathcal{N}(0, 0.1) $.

	
	\begin{figure}[H]
		\centering
				\begin{subfigure}{0.32\textwidth}
			\includegraphics[scale=0.32]{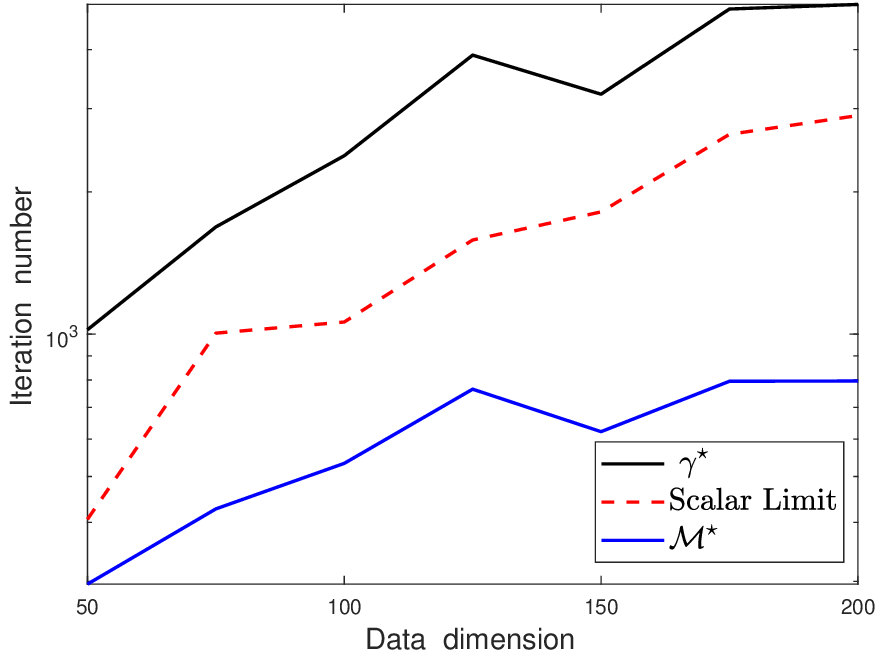}\caption{Scalability.}
			\label{02}
		\end{subfigure}
		\begin{subfigure}{0.32\textwidth}
			\includegraphics[scale=0.32]{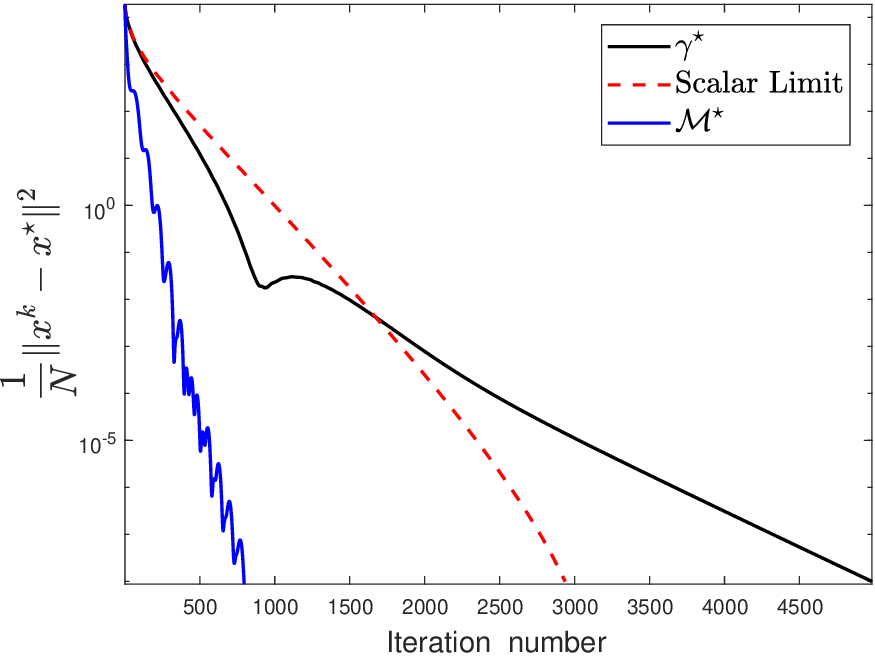}\caption{Error decreasing rate, with  $ N= 200  $}
	\label{a2}
	\end{subfigure}
		\begin{subfigure}{0.32\textwidth}
			\includegraphics[scale=0.32]{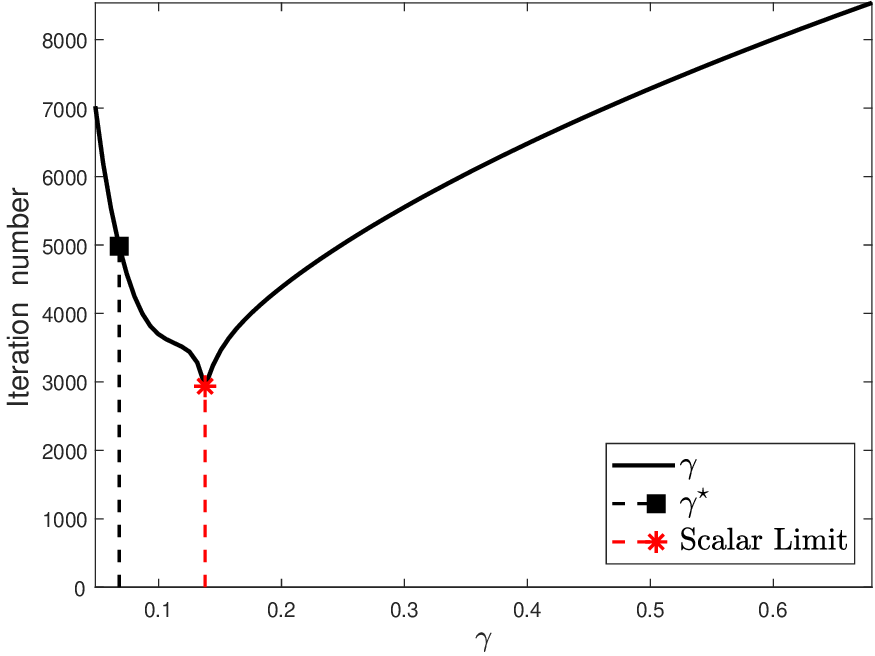}\caption{Scalar limit search.}
	\label{b2}	
	\end{subfigure}
		\caption{Iteration number complexity performance.}
		\label{fig3}
	\end{figure}

	\noindent\textbf{Performance:} 
	In Figure \ref{02}, we evaluate the scalability property of our solver, without manually promoting an ill-conditioning structure. We observe a stable, roughly $ 3\times $ iteration-complexity advantage compared to the scalar limit.
	In Figure \ref{a2}, we investigate the convergence rates  
	with data size $ N = 200 $. We observe (roughly)  linear rates. Particularly, the metric-quipped one significantly outperforms the best scalar parameter, obtained by exhaustive search.  In Figure \ref{b2}, we check the exhaustive search process.  We observe that the theoretical optimal scalar (minimized a worst-case convergence rate) is close to the scalar case limit, meanwhile  the  limit is relatively far away from  the unit case  $\gamma = 1$ (implying the necessity of parameter tuning), and changes with different data sizes (in our additionally experiments).

	\section{Conclusion}
In  this paper,  we present the first metric-equipped ADMM  for semidefinite programming  with  a worst-case performance guarantee. 
Equipping a  metric parameter is new, owing to the challenges: (i)   ADMM $ \bm{Z} $-iterate  (which handles the  positive  semidefinite constraint) no longer admits a closed-form expression; (ii) the optimal choice is open. 
We addressed these two issues in this paper. 
Additionally, we analyse the worst case of our metric-ADMM. In theory,  the worst case corresponds to when the optimal metric reduces to an optimal scalar. Moreover, we 
identify the data structure that causes it. In practice, the worst case can be interpreted as simple-structured data, i.e.,  all data generated in the exact same way. 
Numerically, we observe limited advantage in the worst case compared to employing a scalar parameter, but gains  rapidly increased extra efficiency   as the conditioning becomes worse. 
Across all different  settings, our metric-based solver shares   similar performances (implying the underlying ill-conditioning structure fully exploited), and the iteration number increases very slowly with data dimension.

	\bibliographystyle{unsrt}
	\bibliography{Reference/ref1,Reference/Ref_S,Reference/ML_application,Reference/sr_applications}
	
\end{document}